\documentclass[11pt,twoside]{article}
\usepackage{amsfonts,amssymb}
\usepackage{amsmath}
\usepackage{balance}

\usepackage{amssymb}
\usepackage{amsthm}
  \newtheorem{thm}{Theorem}[section]
  \newtheorem{lem}[thm]{Lemma}
  \newtheorem{prop}[thm]{Proposition}
   \newtheorem{prob}[thm]{Problem}
  \newtheorem{cor}[thm]{Corollary}

  \theoremstyle{definition}
  \newtheorem{defn}{Definition}[section]
  \newtheorem{exmp}{Example}[section]

  \theoremstyle{remark}
  \newtheorem*{rem}{Remark}

\newcommand{\abs}[1]{\left|#1\right|}
\title{A Reshetnyak type theorem for quasiregular values on Carnot groups of $H$-type}
\author{Deguang Zhong \\Institute of Applied Mathematics, Shenzhen Polytechnic University,\\
Shenzhen, Guangdong, 518055, P. R. China\\}
\date{}

\chardef\bslchar=`\\ % p. 424, TeXbook

\providecommand{\qedsymbol}{\leavevmode
  \hbox to.77778em{%
  \hfil\vrule
  \vbox to.675em{\hrule width.6em\vfil\hrule}%
  \vrule\hfil}}

\catcode`\|=0
\begingroup \catcode`\>=13 % in LaTeX2e verbatim env makes > active
\gdef\?#1>{{\normalfont$\langle$\textit{#1}$\rangle$}}
\gdef\0{\relax}
\endgroup
\def\<#1>{{\normalfont$\langle$\textit{#1}$\rangle$}}

\def\latex/{{\protect\LaTeX}}

\usepackage{url}
\usepackage[breaklinks]{hyperref}

\begin{document}
\maketitle
%\markboth{Using the \pkg{amsthm} package}{Using the \pkg{amsthm} package}
\begin{abstract}
In this paper, we establish a Reshetnyak type theorem for quasiregular values on the setting of Carnot group of $H$-type.
\end{abstract}
\begingroup
\small
\tableofcontents
\endgroup

%%%%%%%%%%%%%%%%%%%%%%%%%%%%%%%%%%%%%%%%%%%%%%%%%%%%%%%%%%%%%%%%%%%%%%%%

\newpage %%%%%%%%%%%%%%%%%%%%

\section{Introduction}
 A  deep theorem of Reshetnyak states that each nonconstant quasiregular map in a domain of $\mathbb{R}^{n}$ is discrete, open and sense-preserving. Recently, Kangasniemi and  Onninen \cite[Theorem 1.2]{KO222} extended this result to the case of quasiregular value on Euclidean space. Our work of this paper showed that there is still valid an analogical Reshetnyak type theorem for quasiregular value on the setting of Carnot group of $H$-type.  Actually, we did a bit more. We first introduce a kinds of generalization of quasiregular maps, which is called the generalized finite distortion maps, on  Carnot group.  Then, we established the continuity for a subclass of generalized finite distortion maps on Carnot group of $H$-type. This allows us to use the tool of topological degree to study the Sobolev maps which have quasiregular values.  
  For the continuity of quasiregular values or generalized finite distortion mappings on Euclidean space have been established by  Kangasniemi and  Onninen \cite{KO221},  and 
 Dole\v{z}alov\'{a},  Kangasniemi and  Onninen \cite{DKO24}. The  result we obtained in Theorem \ref{6666312} is non-Euclidean. Actually, we further showed it is locally H\"{o}lder continuous and the explicit H\"{o}lder exponent was obtained, see Theorem \ref{thm20}. It is possibly of independent interest in PDEs and geometry analysis in Carnot group of $H$-type. Our first result of this paper is read as follows.
 \begin{thm}\label{6666312}
 Suppose that  $1/p+1/q<1$ with $1<p,q\leq\infty$ and $\mathbb{G}$ is a Carnot group of $H$-type. Let $\Omega\subset \mathbb{G}$ be a domain and $f:\Omega\rightarrow \mathbb{G}$ be a mapping of class $W_{loc}^{1,Q}(\Omega)$ satisfying  the following inequality
\begin{equation}\label{----1}
\vert D_{h}f(x)\vert^{Q}\leq K(x)J_{f}(x)+\Sigma(x)
\end{equation}
for almost every $x\in\Omega.$ If $K\in L_{loc}^{p}(\Omega),$ $\Sigma/K\in L_{loc}^{q}(\Omega)$ and $K$ satisfies $(\mathcal{K},Q-1)$-spherical condition on $\Omega,$
 then $f$ has a continuous representative. 
\end{thm}
   According to Theorem \ref{6666312}, we can deduce that the mapping mentioned in Theorem \ref{00990099199} below is continuous; see Corollary \ref{cor20}.   This allows us to use the tool of topological degree
  to establish the following Reshetnyak type theorem for quasiregular values on Carnot group of $H$-type. 
  \begin{thm}\label{00990099199}
Let $K\geq1$ be a given constant, $\Omega\subset\mathbb{G}$ be a domain of Carnot group $\mathbb{G}$ of $H$-type  and $y_{0}\in\mathbb{G}.$ Suppose that $f:\Omega\rightarrow \mathbb{G}$ is an nonconstant mapping of class $W_{loc}^{1,Q}(\Omega)$  satisfying  the following inequality
\begin{equation}\label{---1}
\vert D_{h}f(x)\vert^{Q}\leq KJ_{f}(x)+\Sigma(x)(\rho\circ l_{y_{0}^{-1}}\circ f(x))^{Q}
\end{equation}
for almost every $x\in\Omega.$ If $\Sigma\in L_{loc}^{p}(\Omega)$ for some $p>1,$  then $f^{-1}\{y_{0}\}$ is discrete, the local index $i(x,f)$ is positive for every $x\in f^{-1}\{y_{0}\}$, and for every neighborhood $V$ of every  $x\in f^{-1}\{y_{0}\},$  there has $y_{0}\in {\rm int}f(V).$
\end{thm}
It is worth noting that  it still allows for $\Omega$ to have regions where $J_{f}$  is negative for Sobolev maps satisfying the inequality (\ref{---1}). This is one of reasons to cause the main difficulty to establish the discreteness, openness and sense-preserving type results for quasiregular values in Theorem \ref{00990099199}. The analogous results in the setting of Euclidean space were firstly established by Kangasniemi and  Onninen \cite[Theorem 1.2]{KO222}. The works in this paper on establishing most of those properties were inspired by \cite{KO222}.  The main difference is that: in \cite{KO222},  a Logarithmic isoperimetric inequality was established to derive the positivity of the degree; see \cite[Lemma 4.1]{KO222}. The establishment of this Logarithmic isoperimetric inequality heavily relies  on Sobolev embedding theorem on spheres on Euclidean space. To the best of our knowledge, however, there is no directly applicable of Sobolev embedding theorem on spheres in the context of Carnot group of $H$-type. Thus, the novelty of this paper is that we use the higher-integrability for quasiregular values on Carnot group of $H$-type to overcome this difficulty.
 
  $\mathbf{The\;development \;on\;quasiregular\;maps}.$ From the perspective of mappings, the concept of generalized finite distortion mapping stems from the holomorphic mapping $f$, which satisfies the equation $\vert Df(x)\vert^{2}=J_{f}(x)$ on its domain in $\mathbb{C}$ and has many beautiful properties in complex analysis.  However, the classical Liouville theorem tells us that the analogous holomorphic mappings $f$ in Euclidean space, namely satisfy $\vert Df(x)\vert^{n}=J_{f}(x)$, are extremely rare. They are just the restriction of a M\"{o}bius transformation on their domain. One kinds of more unrestricted mappings, which called quasiregular maps, in Euclidean space was introduced by Yu. G. Reshetnyak in 1967 \cite{Res67}, with significant early contributions by Martio, Rickman and V\"{a}is\"{a}l\"{a} 
  \cite{MRV69,MRV70,MRV71}. Its definition is given by the following.
  \begin{defn}{\rm\cite{Res67}}\label{019001}
  For a given constant $K\geq1$ and $n\geq2.$ Suppose that $\Omega\subset\mathbb{R}^{n}$  is an open connected set. Then, a mapping $f:\Omega\rightarrow\mathbb{R}^{n}$ is called quasiregular or a mapping of bounded distortion if it belongs to $W_{loc}^{1,n}(\Omega, \mathbb{R}^{n})$ and satisfies the following inequality 
  $$\vert Df(x)\vert^{n}\leq KJ_{f}(x)$$
  for almost every $x\in\Omega.$
  \end{defn}
We recommend the  interested readers to see  the work of  Ahlfors \cite{Ahl35} and  Gr\"{o}tzsch \cite{ Grö28}, where the theory of quasiregular maps originates from the planar setting. More properties on quasiregular mappings can also be found in the excellent literatures \cite{Res89, Ric93, Vuo88}.  The quasiregular mappings on other non-Euclidean space were also studied, c.f. \cite{FLP16, HH97, GW16, GGWX24, OR09}. One of its generalization, which called quasiregular curves were recently studied by Pankka \cite{Pan20} on the setting of Riemannian manifolds and by Ikonen \cite{Iko23} where the range space is a singular spaces. If we require that the constant $K$ in the Definition \ref{019001} is replaced by a measure function $K:\Omega\rightarrow[1,\infty),$  then it leads to the definition of finite distortion mappings; c.f. \cite{HK14, IM01}.

 Let $K\geq1$ be a given constant.  Suppose that $\Omega\subset\mathbb{R}^{n},n\geq2,$  is a domain  and $\Sigma:\Omega\rightarrow[0,\infty)$ is a  measure function. Then, a mapping $f:\Omega\rightarrow\mathbb{R}^{n}$ belongs to the class of Sobolev space $W_{loc}^{1,n}(\Omega, \mathbb{R}^{n})$ and satisfies the following inequality 
 \begin{equation}\label{666y1}
\vert Df(x)\vert^{n}\leq KJ_{f}(x)+\Sigma(x)\vert f(x)\vert^{n}
\end{equation}
  for almost every $x\in\Omega,$ was investigated by Astala, Iwaniec and Martin \cite[Section 8.5]{ATM09}. The inequality  (\ref{666y1}) was called the heterogeneous distortion inequality in \cite{KO221}.

The H\"{o}lder regularity of planar maps satisfying  (\ref{666y1}), where the term $\Sigma\vert f\vert^{2}$  is replaced by a constant $K_{1}\geq0,$ 
 had already been studied by Nirenberg \cite{Nir53}, Finn and Serrin \cite{FS58}, and Hartman \cite{Har58}.  Simon \cite{Sim77} had also studied the analogous maps between surfaces in Euclidean space. Actually, the 2D solutions for mappings satisfying inequality (\ref{666y1}) has also attract by many interests. See, for example, the excellent works of Astala and P\"{a}iv\"{a}rinta \cite{AP06} on solving the Calder\'{o}n problem.
 
 Recently, the continuity and Liouville theorem  for mappings satisfying inequality (\ref{666y1}) in space were established by Kangasniemi and  Onninen \cite{KO221}. Later, they established the Reshetnyak’s theorem \cite{KO222}, Rickman’s Picard Theorem \cite{KO242},  linear distortion and rescaling properties \cite{KO241} for quasiregular values which defined by below.
 \begin{defn}\cite[Definition 1.1]{KO222} 
   Let $K\geq1$ be a given constant. Suppose that $\Omega\subset\mathbb{R}^{n}$  is a domain, $y_{0}\in\mathbb{R}^{n}$ and $\Sigma\in L_{loc}^{p}(\Omega)$ for some $p>1.$  Then, a mapping $f:\Omega\rightarrow\mathbb{R}^{n}$ has a $(K,\Sigma)$-quasiregular value at $y_{0}$ if it belongs to $W_{loc}^{1,n}(\Omega, \mathbb{R}^{n})$ and satisfies the following inequality
  \begin{equation}
  \vert Df(x)\vert^{n}\leq KJ_{f}(x)+\Sigma(x)\vert f(x)-y_{0}\vert^{n}
  \end{equation}
  for almost every $x\in\Omega.$
  \end{defn}

 More recently, Dole\v{z}alov\'{a},  Kangasniemi and Onninen \cite{DKO24} introduced and established the continuity for the Sobolev maps $f\in W_{loc}^{1,n}(\Omega,\mathbb{R}^{n})$ in Euclidean space which satisfying the following inequality
\begin{equation}\label{000032}
\vert Df(x)\vert^{n}\leq K(x) J_{f}(x)+\Sigma(x)
\end{equation}
for almost every $x\in\Omega.$ Here $K:\Omega\rightarrow[1,\infty)$ and $\Sigma:\Omega\rightarrow[0,\infty)$ are the measurable functions. The author \cite{Zho25} generalized this kinds of mappings satisfying inequality (\ref{000032}) to generalized finite distortion curves from Euclidean domain to Riemannian manifolds. Based on the developments for quasiregular maps and its generalizations discussed above, in this article we introduce the concept for analogous mappings satisfying inequality (\ref{000032}) on Carnot group; see Definition \ref{df1}. 

$\mathbf{The\;development \;on\;Reshetnyak’s\;theorem}.$ 
In nonlinear elasticity theory, analyzing the topological properties for distortion finite mappings $f$ in $\Omega\subset \mathbb{R}^{n}$ with distortion function $K$ has vital significance. A deep theorem of Yu. G. Reshetnyak \cite{Res89} states that each quasiregular map in a domain of $\mathbb{R}^{n}$ is discrete, open and sense-preserving. J. M. Ball \cite{Bal77,Bal81} showed that the quasiregularity is too strong for some problems of this area. Hence, it is natural to consider the case for $K\in L_{loc}^{p}$ for some $p>0.$ In \cite{Bal81}, J. M. Ball shown that there exists a finite distortion mapping with distortion function $K\in L_{loc}^{p},$ where 
$p<n-1,$  is not discrete. In 1993, Iwaniec and \v{S}ver\'{a}k \cite{IS93} proved that the Sobolev map $f\in W_{loc}^{1,2}(\Omega,\mathbb{C})$ with integrable distortion function is discrete and open. They also conjectured that  a Sobolev map $f\in W_{loc}^{1,n}(\Omega,\mathbb{R}^{n})$ is discrete and open if its distortion function $K\in L_{loc}^{n-1}(\Omega).$ Heinonen and Koskela \cite{HK93} showed that a quasi-light 
map $f\in W_{loc}^{1,n}(\Omega,\mathbb{R}^{n})$ is discrete and open if its distortion function $K\in L_{loc}^{p}(\Omega),$ where $p>n-1.$ 
Villamor and Manfredi \cite{VM98}  showed that above quasi-lightness can be removed.  Bj\"{o}rn \cite{Bjö} extended the result of Villamor and Manfredi \cite{VM98} to finite distortion map with its distortion function belongs to certain Orlicz space. Hencl and Mal\'{y} \cite{HM02} showed that a quasi-light mapping $f\in W_{loc}^{1,n}(\Omega,\mathbb{R}^{n})$ with its distortion function belongs to $L_{loc}^{n-1}(\Omega)$  is discrete and open. 
 Unexpectedly, as showed by Hencl and Rajala  \cite{HR13}, the conjecture of Iwaniec and \v{S}ver\'{a}k \cite{IS93} concerning discreteness is false even for Lipschitz mapping  with finite distortion belongs to $L_{loc}^{n-1}(\Omega).$  Recently, Kangasniemi and Onninen  \cite{KO222} established a single-value version of Reshetnyak's theorem. Namely, they showed that a non-constant continuous map $f\in W_{loc}^{1,n}(\Omega,\mathbb{R}^{n})$ satisfying the inequality
$$\vert Df(x)\vert^{n}\leq KJ_{f}(x)+\Sigma(x)\vert f(x)-y_{0}\vert^{n},$$
where $K\geq1$ is a constant, $y_{0}\in \mathbb{R}^{n}$ and $\Sigma\in L_{loc}^{1+\varepsilon}(\Omega),$ then $f^{-1}\{y_{0}\}$ is discrete, the local index $i(x,f)$ is positive in $f^{-1}\{y_{0}\}$, and every neighborhood of a point of $f^{-1}\{y_{0}\}$ is mapped to a neighborhood of $y_{0}.$ 

To extend the theorem of Reshetnyak \cite{Res89}  to the case of non-Euclidean spaces have also aroused  widespread interest.  Heinonen and  Holopainen \cite[Theorem 6.17]{HH97} showed that
for a Carnot group $\mathbb{G}$ of $H$-type, the quasiregular map $f:U\rightarrow\mathbb{G}$ is either constant or a sense-preserving, discrete, and open map.   Zapadinskaya \cite{Zap08} established the discreteness and openness for finite distortion mappings on Carnot groups of $H$-type  with distortion functions belong to certain Orlicz spaces.   Recently, Basalaev and  Vodopyanov \cite{BV23} proved that the finite distortion $f:\Omega\rightarrow \mathbb{G},$ where $ \mathbb{G}$ is a  Carnot group of $H$-type, is continuous, open, and discrete if its  the distortion function 
$K\in L_{loc}^{p}(\Omega), p>Q-1.$ In \cite{Kir16}, Kirsil\"{a}  established  Reshetnyak’s theorem to the case for mapping $f:X\rightarrow\mathbb{R}^{n},$ where $X$ is a generalized $n$-manifold satisfying assumptions such as Ahlfors $n$-regularity and Poincar\'{e} inequality. More recently, Meier and  Rajala \cite[Theorem 1.2]{MR25} extended Iwaniec-\v{S}ver\'{a}k theorem \cite{IS93} to the map $f: X\rightarrow\mathbb{R}^{2},$ where $X$ is a metric surface.

In the rest of this paper are arranged is as follows: In Sect. \ref{sec2222}, we will recall some notations and well known results on Carnot groups which  needed in this paper.  In Sect. \ref{sec255555},
we will establish the continuity for some subclass of generalized finite distortion maps on Carnot group $\mathbb{G}$ of $H$-type.  Actually, we showed that those mappings are locally H\"{o}lder continuous on Carnot group $\mathbb{G}$ of $H$-type and the explicit H\"{o}lder exponent were given. Some problems on  sharp and asymptotically sharp H\"{o}lder exponent were also discussed at the end of section. The Section \ref{sec3333} is devote to proving the discreteness for quasiregular values. In most case, we provide the auxiliary results on some subclass of generalized finite distortion maps in this section.
 In Sect. \ref{sec6}, the positivity of degree for quasiregular values $f$  was established. In Sect. \ref{sec7},  the openness type result for quasiregular values $f$  was obtained.

\section{Preliminary}\label{sec2222}
In this section, we will recall some facts on Carnot group. It should be pointed out that  $C(k_{1},k_{2},\ldots,k_{m})>0$ in this paper will be a constant that only depends on $k_{1},k_{2},\ldots,k_{m}$ but may vary from line to line.

 In this paper, a Carnot group $\mathbb{G}$  \cite{FS82, Pan89} is a connected, simply connected nilpotent Lie group with graded Lie algebra 
$\mathfrak{g}=V_{1}\oplus\cdots\oplus V_{m},$ where the vector subspaces $V_{i}$ satisfying $[V_{1},V_{i}]=V_{i+1},$ $i=1,\cdots,m-1,$ and 
$[V_{1},V_{m}]=0$ and ${\rm dim}\;V_{1}\geq2.$  Then, the homogeneous dimension of $\mathbb{G}$ is  defined by 
$$Q=\sum_{j=1}^{m}j{\rm dim}V_{j}.$$ 
 If $m=2,$ then $\mathbb{G}$ is called a two-step Carnot group.

Suppose that  $\mathbb{G}$ is a Carnot group with graded Lie algebra $\mathfrak{g}=V_{1}\oplus\cdots\oplus V_{m}.$ Then its subspace $V_{1}$ is a horizontal space of $\mathfrak{g},$ and the elements of $V_{1}$ are called the horizontal vector fields. Assume that ${\rm dim}V_{1}=n$ and fix a basis $X_{1},\cdots,X_{n}$ of $V_{1}$ and a left-invariant Riemannian metric $g$ on $\mathbb{G}$ such that $g(X_{i},X_{j})=\delta_{ij}.$ By extending $X_{1},\cdots,X_{n}$ to an orthonormal basic $X_{1},\cdots,X_{n}, T_{1},\cdots,T_{N-n}$ of $\mathfrak{g},$ where $N={\rm dim}\mathbb{G}$
is the topological dimension of $\mathbb{G}.$ Then, it is well known that the exponential map $\exp:\mathfrak{g}\rightarrow\mathbb{G}$ is a global diffeomorphism \cite{FS82}. Hence, an element $g\in \mathbb{G}$ will be identified with a point $(x_{1},\cdots,x_{N})\in \mathbb{R}^{N}$ if $$g=\exp\left(\sum_{i=1}^{n}x_{i}X_{i}+\sum_{i=n+1}^{N}x_{i}T_{i-n}\right).$$
In additional, it gives the bi-invanriant Haar measure on $\mathbb{G}$ by the push-forward of the Lebesgue measure of $\mathbb{R}^{N}$ by using of this exponential map. In this paper, we donate the  bi-invanriant Haar measure on $\mathbb{G}$ by $dx.$ Then, the symbal $m(E)$ stands for this bi-invanriant Haar measure of a set $E.$ 

An absolutely continuous path $\gamma:[0,1]\rightarrow\mathbb{G}$ is said to be horizontal if there exist real-value measurable functions $b_{1},b_{2},\cdots,b_{n}$ such that 
$$\dot{\gamma}(t)=\sum_{j=1}^{n}b_{j}(t)X_{j}(\gamma(t))$$
for almost every $t\in[0,1].$ Its length is defined by 
$$l(\gamma)=\int_{0}^{1}\left(\sum_{j=1}^{n}b_{j}^{2}(t)\right)^{\frac{1}{2}}dt.$$
Thus, the Carnot-Carath\'{e}odory distance $d_{c}(x,y)$ between two different points $x,y\in \mathbb{G}$ is defined by 
$$d_{c}(x,y)=\inf_{\gamma}l(\gamma),$$
where the infimum takes over all horizontal paths $\gamma$ with $\gamma(0)=x$ and  $\gamma(1)=y.$
It provides a left-invariant  Carnot-Carath\'{e}odory metric on $\mathbb{G};$ see \cite{Cho39, Str86}. With this distance, we donate $B(x,r)=\{y\in \mathbb{G}:d_{c}(x,y)<r\}$ the set of ball with radius $r$ which centered at $x.$ And the symbol $S(x,r)$ stands for its boundary. Namely,   $S(x,r)=\{y\in \mathbb{G}:d_{c}(x,y)=r\}.$ Define the dilation $\delta_{\lambda}$ on $\mathbb{G}$ by 
$$\delta_{\lambda}(x)=(\lambda x_{1},\lambda^{2}x_{2},\cdots,\lambda^{m}x_{m}).$$ Here $x=(x_{1},\cdots,x_{j},\cdots,x_{m}), x_{j}\in V_{j},j=1,2,\cdots,m.$
 They are the group automorphisms of  $\mathbb{G}$ for all $\lambda>0.$ In this paper, we normalize the condition on the bi-invanriant Haar measure in $\mathbb{G}$ by letting that the balls with radius 1 satisfying 
\begin{equation}\label{nc}
m(B(x,1))=\int_{B(x,1)}dx=1.
\end{equation}
Since the Jacobian determinant of $\delta_{\lambda}$ equals $r^{Q},$  we get by combing the normalized condition (\ref{nc}) that $m(B(x,r))=r^{Q}.$

 A homogeneous norm on $\mathbb{G}$ is defined by a continuous mapping $\rho:\mathbb{G}\rightarrow[0,\infty),$ which is smooth on $\mathbb{G}\setminus\{0\}$ and satisfies the following conditions:

($a$) $\rho(x)=0$ if and only if $x=0;$

($b$)  $\rho(x^{-1})=\rho(x)$ and $\rho(\delta_{\lambda}x)=\lambda\rho(x).$

It is well known that there exists $c_{1}\geq1,$ such that 
\begin{equation}\label{00oo}
\rho(xy)\leq c_{1}(\rho(x)+\rho(y))
\end{equation} for all $x,y\in\mathbb{G},$ and all homogeneous norm on $\mathbb{G}$ are equivalent; c.f. \cite{FS82}. That is to say there exist constants 
$0<C_{1}\leq C_{2}<\infty,$ such that for any two homogeneous norm $\rho_{1},\rho_{2}$ on $\mathbb{G},$ we have that $C_{1}\rho_{2}(x)\leq\rho_{1}(x)\leq C_{2}\rho_{2}(x).$ We donate the distance
 induced by a norm $\rho$ between two point $x,y\in\mathbb{G}$  as $$d_{\rho}(x,y)=\rho(y^{-1}x).$$

$\mathbf{Absolutiely\;continuous\;on\;lines}.$  Let $\mathbb{G}_{1}$ and $\mathbb{G}_{2}$ be two Carnot groups. Then a mapping $\phi:\Omega\subset\mathbb{G}_{1}\rightarrow\mathbb{G}_{2}$ is said to be absolutely continuous on lines if for every domain $U\subset\Omega$ with $\overline{U}\subset\Omega$ and for the fibering $\Gamma_{j},$ which generated by $X_{j},j=1,\cdots,n,$ $\phi$ is absolutely continuous on $\gamma\cap U$ for 
$d\gamma$-almost every curve $\gamma\subset\Gamma_{j}.$ It is well known that such mapping $\phi$ has the derivetives $X_{j}\phi$ along $X_{j},j=1,\cdots,n,$ for almost everywhere on $\Omega;$ \cite[Proposition 4.1]{Pan89}. 

$\mathbf{Sobolev\;space\;on\;Carnot\;group}.$  
We use the definition of Sobolev space on Carnot group from \cite{Vod23}. Suppose that $\mathbb{G}$ is a Carnot group with graded Lie algebra $\mathfrak{g}=V_{1}\oplus\cdots\oplus V_{m}$ and ${\rm dim}V_{1}=n.$
Let $\Omega$ be a domain in $\mathbb{G}.$ For $p\geq1,$ the symbol $L^{p}(\Omega)$ stands for the class consisting of all the $p$-integrable mappings 
$f:\Omega\rightarrow\mathbb{R}.$ The norm of a mapping $f\in L^{p}(\Omega)$ is defined by 
$$\| f\|_{L^{p}(\Omega)}=\left(\int_{\Omega}\vert f(x)\vert^{p}dx\right)^{\frac{1}{p}}.$$
We say $f\in L_{loc}^{p}(\Omega)$ if for each compact set $K\subset\Omega,$ we have $f\in L^{p}(K).$

 A real-valued function $f$ defined on a domain $\Omega\subset\mathbb{G}$  is said to lie in the Sobolev space $W^{1,p}(\Omega)$ with $p\geq1$ if $f$ is absolutely continuous on lines on $\Omega$ and its norm of $f\in W^{1,p}(\Omega)$ is defined by 
 $$\| f\|_{W^{1,p}(\Omega)}=\left(\int_{\Omega}\vert f(x)\vert^{p}dx\right)^{\frac{1}{p}}+\left(\int_{\Omega}\vert \nabla_{h}f(x)\vert^{p}dx\right)^{\frac{1}{p}}$$
is finite, where the symbol  $\nabla_{h}f=(X_{1}f,\cdots,X_{n}f)$ stands for the subgradient of $f.$ A mapping $f:\Omega\rightarrow \mathbb{R}$ is said to be lied in the local Sobolev space $W_{loc}^{1,p}(\Omega)$ with $p\geq1$ if  for each compact set $K\subset\Omega,$ we have $f\in W^{1,p}(K).$

Now, a mapping $\varphi$ defined between a domain $\Omega\subset \mathbb{G}_{1}$ and a Carnot group $\mathbb{G}_{2}$ is said to be lied in $W_{loc}^{1,p}(\Omega,\mathbb{G}_{2})$ if $\rho_{2}\circ\varphi\in L_{loc}^{p}(\Omega)$ and $\varphi$ is absolutely continuous on lines in $\Omega$ and $X_{j}\varphi\in L_{loc}^{p}(\Omega)$ for $j=1,\cdots,n.$ Here, $\rho_{2}$ is the homogeneous norm on the Carnot group $\mathbb{G}_{2}.$

$\mathbf{Generalized\;finite\;distortion\;maps\;on\;Carnot\;group}.$  
Let $\Omega\subset\mathbb{G}$ be a domain and $f\in W_{loc}^{1,Q}(\Omega,\mathbb{G}).$ Then the symbol $D_{h}f$  defined by $D_{h}f=(X_{i}f_{j})$ with $i,j=1,\cdots,n$  is called the formal horizontal differential of $f.$ It was proved by Pansu in \cite[Proposition 4.1]{Pan89} that the linear operator induced
by  formal horizontal differential $D_{h}f$ is a mapping of the horizontal space $V_{1}$ into the horizontal space $V_{1}.$ In addition, by a result of 
Vodoptyanov \cite{Vod96}, the formal horizontal differential $D_{h}f:V_{1}\rightarrow V_{1}$ generates the homomorphism $Df:\mathfrak{g}\rightarrow \mathfrak{g}.$ The induced homomorphism is called the formal differential of $f.$ Its determinant is called the (formal) Jacobian and denoted by $J_{f}.$ Also, it is well known that the norms for those two differentials satisfying the following inequality
$$\vert D_{h}f(x)\vert\leq\vert Df(x)\vert\leq C\vert D_{h}f(x)\vert,$$
 where $C$ is a constant depends only on the structure of the groups; c.f. \cite{BV23}. Now, we can process to introduce the definition of generalized finite distortion maps on Carnot group.
\begin{defn}\label{df1}
Suppose that $\mathbb{G}$ is a Carnot group and $ \Omega\subset\mathbb{G}$ is a domain.  Then, $f\in W^{1,Q}_{loc}(\Omega,\mathbb{G})$
is said to be a generalized finite distortion map if it satisfies the following inequality
\begin{equation}
\vert D_{h}f(x)\vert^{Q}\leq K(x)J_{f}(x)+\Sigma(x)
\end{equation}
for almost every $x\in\Omega.$ Here $K:\Omega\rightarrow[1,\infty)$ and $\Sigma:\Omega\rightarrow[0,\infty)$ are the measurable functions.
\end{defn}

 The following concept of quasiregular value is analogical to the one in the setting of Euclidean space given  by \cite[Definition 1.1]{KO222}.
\begin{defn}\label{df2}
Let  $K\geq1$ be a given constant. Suppose that $ \Omega\subset\mathbb{G}$ is a domain and $y_{0}\in\mathbb{G}.$  Then, $f\in W^{1,Q}_{loc}(\Omega,\mathbb{G})$
is said to be having a $(K, \Sigma)$-quasiregular value at $y_{0}$ if it satisfies the following inequality
\begin{equation}\label{df2111111}
\vert D_{h}f(x)\vert^{Q}\leq KJ_{f}(x)+\Sigma(x)(\rho\circ l_{y_{0}^{-1}}\circ f(x))^{Q}
\end{equation}
for almost every $x\in\Omega.$ Here $\Sigma:\Omega\rightarrow[0,\infty)$ is a measurable function.
\end{defn}

\begin{rem} 
It is worth noting that in Definitions \ref{df1} and \ref{df2}, it still allows for $\Omega$ to have regions where $J_{f}$  is negative.
\end{rem}

\begin{defn}
For a domain $\Omega\subset\mathbb{G},$ we say a function $g:\Omega\rightarrow [1,\infty)$ is $\mathcal{K}_{1}$-bounded $p$-mean integrable on every ball $B:=B(x,r)\subset\Omega$ if
\begin{equation}
\mathcal{K}_{1}:=\sup_{0<r<{\rm dist}(x,\partial\Omega)}\left(\frac{\int_{B}g^{p}(x)dx}{m(B)}\right)^{\frac{1}{p}}<\infty.
\end{equation}
In this situation, we simply say that $g$ satisfies $(\mathcal{K}_{1},p)$-condition on $\Omega.$
\end{defn}

\begin{defn}
For a domain $\Omega\subset\mathbb{G},$ we say a function $g:\Omega\rightarrow [1,\infty)$ is $\mathcal{K}_{2}$-bounded $p$-mean integrable on every sphere $\partial B:=\partial B(x,r)\subset\Omega$ if 
\begin{equation}
\mathcal{K}_{2}:=\sup_{0<r<{\rm dist}(x,\partial\Omega)}\left(\frac{\int_{\partial B}g^{p}(x)d\sigma}{m(\partial B)}\right)^{\frac{1}{p}}<\infty.
\end{equation}
Here we simply say that $g$ satisfies $(\mathcal{K}_{2},p)$-spherical condition on $\Omega.$
\end{defn}

$\mathbf{Carnot\;group\;of\;}$$H$$\mathbf{-type\;and\;its\;norm}.$ The following concept on Carnot group of $H$-type was introduced by Kaplan \cite{Kap80}.
\begin{defn}\cite{Kap80}
We say a Carnot group $\mathbb{G}$ is of $H$-type if its Lie algebra $\mathfrak{g}=V_{1}\oplus V_{2}$ is two step, and if, in addition, there is an inner produce $\langle,\rangle$ in $\mathfrak{g}$ such that for the linear mapping $J:V_{2}\rightarrow{\rm End}\,V_{1}$
satisfying $\langle J_{Z}(U),V\rangle=\langle Z,[U,V]\rangle$
and $J_{Z}^{2}=-\vert Z\vert^{2} \mathbf{Id}$ for all $U,V\in V_{1}$ and $z\in V_{2}.$
\end{defn}

 Since any two  homogeneous norms on  Carnot group are equivalent, it allows us to only consider the following homogeneous norm on  Carnot group $\mathbb{G}$ of $H$-type defined by following ways: For each $x\in \mathbb{G},$ we let $a(x)\in V_{1},b(x)\in V_{2},$ such that $x=\exp\left(a(x)+b(x)\right).$ Then, we define 
\begin{equation}\label{norm1}
\rho(x)=\left(\vert a(x)\vert^{4}+16\vert b(x)\vert^{2}\right)^{1/4}.
\end{equation}
It was proved by Heinonen and Holopainen \cite{HH97} that the norm given by (\ref{norm1}) is smooth and satisfies the following equation 
\begin{equation}
{\rm div}_{h}(\vert\nabla_{h}\gamma\vert)^{Q-1}\nabla_{h}\gamma)=0
\end{equation}
 outside the identity element of $\mathbb{G}.$ Here, $\gamma=\log\rho.$ 

The following example of Heisenberg group is a  primary example of a Carnot group of $H$-type. It can  help us to better understand the properties of Carnot groups of $H$-type.
\begin{exmp}\cite[Example 1.1]{BV23}
Let $\mathbb{H}^{n}=(\mathbb{R}^{2n+1},\ast)$ be a Heisenberg group with homogeneous dimension $Q=2n+2.$ The group law of  $\mathbb{H}^{n}$ is  given by
$$(x,y,t)\ast(x',y',t')=\left(x+x',y+y',t+t'+\frac{x\cdot y'-x'y}{2}\right),$$
where $x,x',y,y'\in \mathbb{R}^{n},t,t'\in\mathbb{R}.$ It is Lie algebra $\mathfrak{g}=V_{1}\oplus V_{2}$ is spanned by the vector fields
$$X_{i}=\frac{\partial}{\partial x_{i}}-\frac{y_{i}}{2}\frac{\partial}{\partial t},Y_{i}=\frac{\partial}{\partial y_{i}}-\frac{x_{i}}{2}\frac{\partial}{\partial t}, i=1,\cdots,n,T=\frac{\partial}{\partial t},$$
where $V_{1}={\rm span}\{X_{i},Y_{i}:i=1,\cdots,n\}$ and $V_{2}={\rm span}\{T\}.$ The only nontrivial Lie brackets are that $[X_{i},Y_{i}]=T,$ where $i=1,\cdots,n.$ Suppose that $\langle \cdot,\cdot\rangle$
is a inner product such that the vector fields are orthonormal. Then, a linear mapping $J_{T}:V_{1}\rightarrow V_{1}$ can be defined by $J_{T}(X_{i})=Y_{i}$ and $J_{T}(Y_{i})=-X_{i}.$ Its homogeneous norm is defined 
\begin{equation}\label{099009}
\rho(p)=\left(\vert x\vert^{4}+t^{2}\right)^{1/4},
\end{equation}
where $p=(x,t)\in \mathbb{H}^{n}.$ A straightforward calculation shows that the function $\gamma(p):=\log\rho(p)$  satisfies the equation 
${\rm div}_{h}(\vert\nabla_{h}\gamma\vert)^{2n+1}\nabla_{h}\gamma)=0,$
where $\rho(p)$ is given by (\ref{099009}).
\end{exmp}

$\mathbf{Degree\;theory}.$  Here we recall some concepts and known results on degree of a map on Carnot group. For detailed information see  \cite{Fed69, FG95, OR09}. For a domain $\Omega$ with closure $\overline{\Omega}\subset\mathbb{G}$ and $f:\Omega\rightarrow\mathbb{G}$ is continuous, we say the  triple $(f,\Omega,y)$ is admissible if $f$ is proper and $y\notin f(\partial\Omega).$ Here,  $f:\Omega\rightarrow\mathbb{G}$ is said to be proper if $f^{-1}\{y\}$ is compact for every $y\in\mathbb{G}.$

Suppose that the   triple $(f,\Omega,y)$ is admissible. If $f$ is a smooth map and $y$ is a regular point of $f.$ Then, we define 
$${\rm deg}(f,\Omega,y):=\sum_{x\in f^{-1}\{y\}\cap\Omega}{\rm sgn}(J_{f}(x)).$$
If $f$ is not smooth or $y$ is not a regular point of $f.$ Then, the definition of  ${\rm deg}(f,\Omega,y)$ of the admissible   triple $(f,\Omega,y)$  can still be defined by the approximation of smooth functions. Namely, it is  defined by 
$${\rm deg}(f,\Omega,y):={\rm deg}(f_{1},\Omega,y_{1}).$$
Here, the  triple $(f_{1},\Omega,y_{1})$ is admissible, for which $f_{1}:\Omega\rightarrow\mathbb{G}$ is smooth, $y_{1}$ is a regular point of $f_{1},$ and satisfying the following inequality
$$d(y,y_{1})+\sup _{x\in \overline{\Omega}}d(f(y),f(y_{1}))\ll {\rm dist}(y,f(\partial\Omega)).$$

The following results on degree of a map can be found in \cite{Fed69, FG95, OR09}.
\begin{lem}\label{lem2.1}
Suppose that $f:\overline{\Omega}\rightarrow\mathbb{G}$ is proper, where $\Omega\subset\mathbb{G}$ is a domain.
\begin{itemize}
\item  If two points $y,y_{1}$ are in the same path component of $\mathbb{G}\setminus f(\partial\Omega),$ then ${\rm deg}(f,\Omega,y)={\rm deg}(f,\Omega,y_{1}).$ In particular,  if $y\in V\subset\mathbb{G}\setminus f(\partial\Omega),$ and $V$ is connected. Then ${\rm deg}(f,\Omega,y)$  is constant-valued on $V,$ which is independent of $y.$  In this case,  we call the constant of ${\rm deg}(f,\Omega,y)$ the degree of $f$
    with respect to $V$ and  denote ${\rm deg}(f,\Omega):={\rm deg}(f,\Omega,y).$ 
\item If $V\subset\Omega$ is an open set and satisfies that $f^{-1}\{y\}\subset V.$ Then, $${\rm deg}(f,V,y)={\rm deg}(f,\Omega,y).$$ 
\item  If $\Omega=\bigcup_{i\in \mathbb{N}}\Omega_{i},$ where $\Omega_{i}$ are open sets for all $i\in \mathbb{N}$ satisfying $\Omega_{i}\cap \Omega_{j}=\emptyset$ for $i\neq j.$ Then,
$${\rm deg}(f,\Omega,y)=\sum_{i\in \mathbb{N}}{\rm deg}(f,\Omega_{i},y).$$
\end{itemize}
\end{lem}

$\mathbf{Local\;index}.$ Let $\Omega$ be a connected domain, $y_{0}\in\mathbb{G}$ and   $f:\Omega\rightarrow\mathbb{G}$ is continuous. Suppose that $x_{0}\in f^{-1}\{y_{0}\}.$ If $V_{1}$ and $V_{2}$ are two neighborhoods of $x_{0}$ such that $\overline{V_{1}}\cap f^{-1}\{y_{0}\}=\overline{V_{2}}\cap f^{-1}\{y_{0}\}=\{x_{0}\},$ then $y_{0}\notin f(V_{i}), i=1,2.$ Hence, it is meaningful to define  ${\rm deg}(f,V_{i},y_{0}),i=1,2.$ In fact, it well known that ${\rm deg}(f,V_{1},y_{0})={\rm deg}(f,V_{2},y_{0})={\rm deg}(f,V_{1}\cup V_{2},y_{0}).$ Therefore, if there exist a neighborhood $V$ of $x_{0},$ such that 
$\overline{V}\cap f^{-1}\{y_{0}\}=\{x_{0}\},$ then the same value of ${\rm deg}(f,V,y_{0})$ is regardless of the choice of $V.$ We call this value the local index of $f$ at $x_{0}$ and  record it as $i(x_{0},f).$

 \section{H\"{o}lder continuity for some subclass of generalized finite distortion maps}\label{sec255555}
In this section, we will establish  the continuity for a subclass of generalized finite distortion maps defined by Definition \ref{df1} in the case that $\mathbb{G}$ is a Carnot group of $H$-type. Actually, the mapping mentioned in Theorem  \ref{6666312} is locally H\"{o}lder continuous. The main result of this section is to give the explicit H\"{o}lder exponent for this kinds of mappings.  At the rest of this paper,  the constants $c_{1}$ and $I$ are the ones given by  (\ref{00oo}) and (\ref{iso00000}), respectively. The main result of this section is read as follows.
\begin{thm}\label{thm20}
Suppose that  $1/p+1/q<1$ with $1<p,q\leq\infty$ and $\mathbb{G}$ is a Carnot group of $H$-type. Let $\Omega\subset\mathbb{G}$ be a bounded domain and $f:\Omega\rightarrow \mathbb{G}$ be a mapping of class $W^{1,Q}(\Omega)$ satisfying  the following inequality
$$\vert D_{h}f(x)\vert^{Q}\leq K(x)J_{f}(x)+\Sigma(x)$$
for almost every $x\in\Omega.$ If $K\in L^{p}(\Omega),$ $\Sigma/K\in L^{q}(\Omega)$ and $K$ satisfies $(\mathcal{K},Q-1)$-spherical condition on $\Omega$, then $f$ has a continuous representative.
In particular,  for every fixed $x\in\Omega$ and for $r$ which satisfies $r<c_{1} R_{0}\ll{\rm dist}(x,\partial\Omega).$

(i)\,If $Q(1-1/p-1/q)\neq1/I\mathcal{K},$  then we have
 \begin{equation}
 \begin{aligned}
d_{c}(f(z),f(y))
\leq &C_{1}(d_{c}(z,y))^{\min\Big\{1-\frac{1}{p}-\frac{1}{q},\frac{1}{QI\mathcal{K}}\Big\}}\\
&\,\,\,\,\,\,\,\,\,\,\,\,\,\,\,\,\,\,\,\,+C_{2}(d_{c}(z,y))^{\frac{1}{2}\min\Big\{1-\frac{1}{p}-\frac{1}{q},\frac{1}{QI\mathcal{K}}\Big\}}\\
\end{aligned}
\end{equation}
for every $y,z\in B(x,r/4).$  Here $C_{1}$ and $C_{2}$ are two positive constants depend only on $Q, p, q,\mathcal{K}, R_{0},$ $
\Vert D_{h}f\Vert_{L^{Q}(\Omega)}, \|K\|_{L^{p}(\Omega)}$ and $\|\Sigma/K\|_{L^{q}(\Omega)}.$  

(ii)\,If $Q(1-1/p-1/q)=1/I\mathcal{K},$ then we have 
  \begin{equation}
\begin{aligned}
d_{c}(f(y),f(z))\leq C_{3}&d_{c}(y,z)^{\frac{1}{QI\mathcal{K}}}\left(\log\frac{c_{1}R_{0}e^{1+I\mathcal{K}}}{d_{c}(y,z)}\right)^{\frac{1}{Q}}\\
&\,\,\,\,\,\,\,\,\,\,\,\,\,\,\,\,\,\,\,\,+C_{4}d_{c}(y,z)^{\frac{1}{2QI\mathcal{K}}}\left(\log\frac{c_{1}R_{0}e^{1+I\mathcal{K}}}{d_{c}(y,z)}\right)^{\frac{1}{2Q}}
\end{aligned}
\end{equation}
for every $y,z\in B(x,r/4).$   Here $C_{3}$ and $C_{4}$ are the positive constants which depend only on $Q, p, q,\mathcal{K}, R_{0},$ $
\Vert D_{h}f\Vert_{L^{Q}(\Omega)}, \|K\|_{L^{p}(\Omega)}$ and $\|\Sigma/K\|_{L^{q}(\Omega)}.$  
\end{thm}
In particular, according to Theorem \ref{thm20},  we obtain the following result in the case that  $K\geq1$ is a constant.
\begin{cor}\label{corroo}
For a given constant $K\geq1.$ Suppose that  $\mathbb{G}$ is a Carnot group of $H$-type, $\Omega\subset \mathbb{G}$ is a domain and $f:\Omega\rightarrow \mathbb{G}$ is a mapping of class $W^{1,Q}(\Omega)$ satisfying  the following inequality
$$\vert D_{h}f(x)\vert^{Q}\leq KJ_{f}(x)+\Sigma(x)$$
for almost every $x\in\Omega.$ If  $\Sigma\in L^{1+\varepsilon}(\Omega)$ for some $\varepsilon>0$, then $f$ has a continuous representative. In particular,  for every fixed $x\in\Omega$ and for $r$ which satisfying $r< c_{1}R_{0}\ll{\rm dist}(x,\partial\Omega).$

(i)\,If $Q\varepsilon/(1+\varepsilon)\neq1/IK,$  then we have
   \begin{equation}
 \begin{aligned}
d_{c}(f(z),f(y))
\leq &C_{1}(d_{c}(z,y))^{\min\Big\{\frac{\varepsilon}{1+\varepsilon},\frac{1}{QIK}\Big\}}+C_{2}(d_{c}(z,y))^{\frac{1}{2}\min\Big\{\frac{\varepsilon}{1+\varepsilon},\frac{1}{QIK}\Big\}}\\
\end{aligned}
\end{equation}
for every $y,z\in B(x,r/4).$  Here $C_{1}$ and $C_{2}$ are two positive constants depend only on $Q, p, q,K, R_{0},$ $
\Vert D_{h}f\Vert_{L^{Q}(\Omega)}$ and $ \|\Sigma\|_{L^{1+\varepsilon}(\Omega)}.$ 

(ii)\,If $Q\varepsilon/(1+\varepsilon)=1/IK,$ then we have 
  \begin{equation}
\begin{aligned}
d_{c}(f(y),f(z))\leq C_{3}&d_{c}(y,z)^{\frac{1}{QIK}}\left(\log\frac{c_{1}R_{0}e^{1+IK}}{d_{c}(y,z)}\right)^{\frac{1}{Q}}\\
&\,\,\,\,\,\,\,\,\,\,\,\,\,\,\,\,\,\,\,\,+C_{4}d_{c}(y,z)^{\frac{1}{2QIK}}\left(\log\frac{c_{1}R_{0}e^{1+IK}}{d_{c}(y,z)}\right)^{\frac{1}{2Q}}
\end{aligned}
\end{equation}
for every $y,z\in B(x,r/4).$   Here $C_{3}$ and $C_{4}$ are the positive constants which depend only on $Q, p, q,K, R_{0},$ $
\Vert D_{h}f\Vert_{L^{Q}(\Omega)}$ and $ \|\Sigma\|_{L^{1+\varepsilon}(\Omega)}.$  
\end{cor}

Next, we will make some preparations to prove Theorem \ref{thm20}. In particular, we need the following isopermetric inequality for Sobolev maps on two-step of Carnot groups, which was established in \cite[Theorem 3.1]{Vod07}.
\begin{lem}{\rm\cite[Theorem 3.1]{Vod07}}\label{iso}
Suppose that $\mathbb{G}$ is a two-step Carnot group. Let $\Omega\subset \mathbb{G}$ be a bounded domain and $f:\Omega\rightarrow \mathbb{G}$ be a mapping of class $W^{1,Q}(\Omega).$ Then, the following inequality 
\begin{equation}\label{iso00000}
\abs{\int_{B(x,r)}J_{f}}\leq I\left(\int_{S(x,r)}\vert D^{\sharp}f \vert\right)^{Q/(Q-1)}
\end{equation}
holds for every $x\in\Omega$ and almost all $r\in(0,{\rm dist}(x,\partial\Omega)).$ Here, $I$ is a constant which independent of the choice of $f.$
\end{lem}

By virtue of Lemma \ref{iso}, we obtain the following result.

\begin{lem}\label{lem0}
 Suppose that  $1/p+1/q<1$ with $1<p,q\leq\infty$ and  $\mathbb{G}$ is a Carnot group of $H$-type. Let $\Omega\subset \mathbb{G}$ be an bounded domain and $f:\Omega\rightarrow \mathbb{G}$ be a mapping of class $W^{1,Q}(\Omega)$ satisfying  the following inequality
$$\vert D_{h}f(x)\vert^{Q}\leq K(x)J_{f}(x)+\Sigma(x)$$
for almost every $x\in\Omega.$ If $K\in L^{p}(\Omega),$ $\Sigma/K\in L^{q}(\Omega)$ and $K$ satisfies $(\mathcal{K}, Q-1)$-spherical condition on $\Omega$,
 then there exists $A=A(\|K\|_{L^{p}(\Omega)},\|\Sigma/K\|_{L^{q}(\Omega)})>0,$ such that the followings inequality
$$\int_{B(x,r)}\frac{\mid D_{h}f\mid^{Q}}{K}\leq I\mathcal{K}r\cdot\int_{S(x,r)}\frac{\mid D_{h}f\mid^{Q}}{K}+A\cdot r^{Q(1-1/p-1/q)}$$
holds for almost every $x\in\Omega$ and for every $r\in(0,{\rm dist}(x,\partial\Omega)).$ In particular, there exists a constant $C=C(Q,  p, q,\mathcal{K},\Vert D_{h}f\Vert_{L^{Q}(\Omega)})>0,$ such that the following results hold:

(i) If $Q(1-1/p-1/q)<1/I\mathcal{K},$ then we have 
$$\int_{B(x,r)}\frac{\abs{ D_{h}f}^{Q}}{K}\leq Cr^{Q(1-1/p-1/q)}$$
for almost every $r\in[0,R],$ where $R<{\rm dist}(x,\partial\Omega).$

(ii) If $Q(1-1/p-1/q)=1/I\mathcal{K},$ then we have 
$$\int_{B(x,r)}\frac{\abs{ D_{h}f}^{Q}}{K}\leq Cr^{Q(1-1/p-1/q)}\log\left(\frac{Re^{1+I\mathcal{K}}}{r}\right)$$
for almost every $r\in[0,R],$ where $R<{\rm dist}(x,\partial\Omega).$

(iii) If $Q(1-1/p-1/q)>1/I\mathcal{K},$ then we have 
$$\int_{B(x,r)}\frac{\abs{ D_{h}f}^{Q}}{K}\leq Cr^{1/I\mathcal{K}}$$
for almost every $r\in[0,R],$ where $R<{\rm dist}(x,\partial\Omega).$

\end{lem}
\begin{proof}
We fix a point $x\in\Omega.$  Then, by the definition of $f,$ we get for almost all $r\in(0,{\rm dist}(x,\partial\Omega)),$
\begin{equation}\label{e1}
\int_{B(x,r)}\frac{\abs{ D_{h}f}^{Q}}{K}\leq\int_{B(x,r) }J_{f}+\int_{B(x,r)}\frac{\Sigma}{K}.
\end{equation}
By the isopermetric inequality for Sobolev maps in Lemma \ref{iso} and the $(\mathcal{K},Q-1)-$spherical condition of $K$ and H\"{o}lder inequality, we obtain
\begin{equation}\label{e2}
\begin{aligned}
\int_{B(x,r) }J_{f}\leq& I\left(\int_{ S(x,r)}\abs{ D^{\sharp}f}\right)^{\frac{Q}{Q-1}}\\
\leq& I\left(\int_{S(x,r)}\abs{D_{h}f}^{Q-1}\right)^{\frac{Q}{Q-1}}\\
\leq& I\left(\int_{S(x,r)}K^{Q-1}\right)^{\frac{1}{Q-1}}\cdot\int_{S(x,r)}\frac{\abs{ D_{h}f}^{Q}}{K}\\
\leq& I\mathcal{K}r\cdot\int_{S(x,r)}\frac{\abs{D_{h}f}^{Q}}{K}\\
\end{aligned}
\end{equation}
for almost all $r\in(0,{\rm dist}(x,\partial\Omega)).$ Let $W>0$ such that $1/p+1/q+1/W=1.$ Then we have
\begin{equation}\label{e3}
\begin{aligned}
&\int_{B(x,r)}\frac{\Sigma}{K}\leq\int_{B(x,r)}\Sigma=\int_{B(x,r)}K\cdot\frac{\Sigma}{K}\cdot1\\
\leq& \left(\int_{B(x,r)}K^{p}\right)^{\frac{1}{p}}\cdot\left(\int_{B(x,r)}\left(\frac{\Sigma}{K}\right)^{q}\right)^{\frac{1}{q}}\cdot\left(\int_{B(x,r)}1^{W}\right)^{\frac{1}{W}}\leq A\cdot r^{\frac{Q}{W}},\\
\end{aligned}
\end{equation}
where $A=A(\|K\|_{L^{p}(\Omega)},\|\Sigma/K\|_{L^{q}(\Omega)})>0.$
If we denote $\rho_{c}(x):=d_{c}(x,0),$
$$\varphi(r)=\int_{B(x,r)}\frac{\mid D_{h}f\mid^{Q}}{K}\;\;\;{\rm and}\;\;\;\phi(r)=\int_{S(x,r)}\frac{\abs{D_{h}f}^{Q}}{K}.$$
 Then,  as $\vert D_{h}f\vert^{Q}/K\leq\vert D_{h}f\vert^{Q}\in L_{loc}^{1}(\Omega)$ and $\vert\nabla_{h}\rho_{c}(x)\vert=1$ with $x=(w,z)\in \mathbb{G}$ and $w\neq0,$ 
 we get by using of coarea formula in \cite[(2.13)]{Vod07} that $\varphi'(r)=\phi(r)$ for almost every $r\in(0,{\rm dist}(x,\partial\Omega)).$ Hence, we obtain the following inequality by combining inequalities (\ref{e1}), (\ref{e2}) and (\ref{e3}),
$$\varphi(r)\leq I\mathcal{K}r\cdot\varphi'(r)+A\cdot r^{\frac{Q}{W}}.$$
Now, the proof of (i)-(iii) follow from \cite[Lemma 3.2]{KO221}. 
\end{proof}

The following Poincar\'{e} inequality in the setting of Carnot group can be found in \cite{Bas14, DGP09, GN96, HK00, Jer86, Lu92}.
\begin{lem}\label{-11}
Suppose that $\Omega\subset\mathbb{G}$ is a bounded domain and $p\geq1$. If $f\in W^{1,p}(\Omega),$ then there exists a constant $C>0,$ such that the following inequality
\begin{equation}\label{0311888888000}
\left(\int_{B(x,r)}\vert f(x)-f_{B(x,r)}\vert^{p}dx\right)^{1/ p}\leq Cr\left(\int_{B(x,r)}\vert \nabla_{h}f(x)\vert^{p}dx\right)^{1/p}
\end{equation}
holds for all $x\in\Omega$ and $r>0$ such that $B(x,r)\subset\Omega.$
\end{lem}
By virtue of Lemma \ref{-11}, we have the following continuous result for real-value functions on a domain of Carnot group of $H$-type.
\begin{prop}\label{09888888111}
Suppose that $\mathbb{G}$ is a Carnot group of $H$-type and $\Omega\subset\mathbb{G}$ is a bounded domain. Let $f:\Omega\rightarrow\mathbb{R}$ be belonging to $W^{1,Q}(\Omega).$ If there exist constants $\alpha>0$ and $\beta\geq0,$ such that the following inequality 
\begin{equation}\label{09}
\int_{B(x,r)}\frac{\abs{ \nabla_{h}f}^{Q}}{K}\leq C r^{\alpha}\left(\log\frac{L}{r}\right)^{\beta}
\end{equation}
holds for almost every $x\in\Omega$ and $0<r<{\rm dist}(x,\partial\Omega)$ and $K$ satisfies  $(\mathcal{K}, Q-1)$-condition on $\Omega,$ then we have
\begin{equation}\label{90oiiiii}
\vert f(y)-f(z)\vert\leq C_{1}(d_{c}(z,y))^{\frac{\alpha}{Q}}\left(C_{2}\cdot\left(\log\frac{L}{d_{c}(z,y)}\right)^{\frac{\beta}{Q}}+C_{3}\right)
\end{equation}
for every $y,z\in B(x,r)$ and $r$ is small enough. Here, $C_{1}=C(C,Q,\mathcal{K},\beta)$ is a constant depends only on $C, Q,\mathcal{K}$ and $\beta,$ 
 $$C_{2}=\frac{e^{\frac{\alpha}{Q}}}{e^{\frac{\alpha}{Q}}-1}\;\;\; {\rm and} \;\;\;C_{3}=\sum_{i=0}^{\infty}\frac{i^{\frac{\beta}{Q}}}{e^{i\frac{\alpha}{Q}}}<\infty.$$
\end{prop}
\begin{proof}
The method to prove this proposition is similar the one given in \cite{DKO24}. Firstly, we consider the Lebesgue points of $f.$ We fix a ball $B=B(x,c_{1}R)$ with $0<c_{1}R<<{\rm dist}(x,\partial\Omega),$ such that 
$B$ is geodesically convex.  Here, $c_{1}$ is the constant in the generalized triangle inequality (\ref{00oo}). We suppose that $A\subset B$ is the set of all Lebesgue points of $f$ in $B$ and then show that $f$ satisfies inequality (\ref{90oiiiii}) on $A\cap B(x,R/4).$ Now, we take every $y,z\in A\cap B(x,R/4).$ Then, there exists a $w\in B(x,R/4),$ such that $d_{c}(z,w)=d_{c}(w,y)=d_{c}(z,y)/2.$ Now, we choose the sequence balls $B_{i}\subset B(x,c_{1}R),i\in\mathbb{Z},$ such that $B_{0}=B(w,r_{0}),0<r_{0}=d_{c}(z,y)<R/2$ and $B_{i}=B(y,e^{-\vert i\vert}r_{0})$ when $i\in \mathbb{Z}_{>0},$ and $B_{i}=B(z,e^{-\vert i\vert}r_{0})$ when $i\in \mathbb{Z}_{<0}.$ Since $y$ and
$z$ are Lebesgue points, we get 
$\lim_{i\rightarrow\infty}f_{B_{i}}=f(y)$ and $\lim_{i\rightarrow-\infty}f_{B_{i}}=f(z),$ where the symbol $f_{B_{i}}$ stands for the integral average of $f$ over $B_{i}.$ Now, for every $i\in\mathbb{Z},$ we get by inequality (\ref{0311888888000}) and H\"{o}lder inequality,

\begin{equation}
\begin{aligned}
&\vert f_{B_{i+1}}-f_{B_{i}}\vert\leq\int_{B_{i+1}}\!\!\!\!\!\!\!\!\!\!\!\!\!\!\!\!\!\; {}-{} \,\,\,\,\,\,\, \vert f-f_{B_{i}}\vert\\
\leq &e^{-Q}\int_{B_{i}}\!\!\!\!\!\!\!\!\!\!\!\!\; {}-{} \,\,\,\,\,\,\, \vert f-f_{B_{i}}\vert\leq e^{-Q}\left(\int_{B_{i}}\!\!\!\!\!\!\!\!\!\!\!\!\; {}-{} \,\,\,\,\,\,\, \vert f-f_{B_{i}}\vert^{Q-1}\right)^{\frac{1}{Q-1}}\\
\leq& C(Q)r_{i}\left(\int_{B_{i}}\!\!\!\!\!\!\!\!\!\!\!\!\; {}-{} \,\,\,\,\,\,\, \vert \nabla_{h}f\vert^{Q-1}\right)^{\frac{1}{Q-1}}\\
= &C(Q)r_{i}\left(\int_{B_{i}}\!\!\!\!\!\!\!\!\!\!\!\!\; {}-{} \,\,\,\,\,\,\, 
 \left(\frac{\vert\nabla_{h}f\vert^{Q}}{K}\right)^{\frac{Q-1}{Q}}K^{\frac{Q-1}{Q}}\right)^{\frac{1}{Q-1}}\\
\leq &C(Q)r_{i}\left(\int_{B_{i}}\!\!\!\!\!\!\!\!\!\!\!\!\; {}-{} \,\,\,\,\,\,\, \frac{\vert \nabla_{h}f\vert^{Q}}{K}\right)^{\frac{1}{Q}}\cdot\left(\int_{B_{i}}\!\!\!\!\!\!\!\!\!\!\!\!\; {}-{} \,\,\,\,\,\,\,  K^{Q-1}\right)^{\frac{1}{Q^{2}-Q}}\\
\leq& C(Q,\mathcal{K})\left(\int_{B_{i}}\frac{\vert \nabla_{h}f\vert^{Q}}{K}\right)^{\frac{1}{Q}}\leq  C(C, Q,\mathcal{K})r_{i}^{\frac{\alpha}{Q}}\left(\log\frac{L}{r_{i}}\right)^{\frac{\beta}{Q}}.\\
\end{aligned}
\end{equation}
Therefore, we have 
\begin{equation}
\begin{aligned}
&\vert f(y)-f(z)\vert\leq \sum_{i=-\infty}^{\infty}\vert f_{B_{i+1}}-f_{B_{i}}\vert\\
\leq& 2C(C,Q,\mathcal{K})\sum_{i=0}^{\infty}r_{i}^{\frac{\alpha}{Q}}\left(\log\frac{L}{r_{i}}\right)^{\frac{\beta}{Q}}\\
=& 2C(C,Q,\mathcal{K})(d_{c}(z,y))^{\frac{\alpha}{Q}}\sum_{i=0}^{\infty}e^{-i\frac{\alpha}{Q}}\left(i+\log\frac{L}{d_{c}(z,y)}\right)^{\frac{\beta}{Q}}\\
\leq& C(C,Q,\mathcal{K},\beta)(d_{c}(z,y))^{\frac{\alpha}{Q}}\sum_{i=0}^{\infty}e^{-i\frac{\alpha}{Q}}\left(i^{\frac{\beta}{Q}}+\left(\log\frac{L}{d_{c}(z,y)}\right)^{\frac{\beta}{Q}}\right)\\
=& C(C,Q,\mathcal{K},\beta)(d_{c}(z,y))^{\frac{\alpha}{Q}}\left(\left(\log\frac{L}{d_{c}(z,y)}\right)^{\frac{\beta}{Q}}\sum_{i=0}^{\infty}e^{-i\frac{\alpha}{Q}}+\sum_{i=0}^{\infty}\frac{i^{\frac{\beta}{Q}}}{e^{i\frac{\alpha}{Q}}}\right)\\
=& C(C,Q,\mathcal{K},\beta)(d_{c}(z,y))^{\frac{\alpha}{Q}}\left(\left(\log\frac{L}{d_{c}(z,y)}\right)^{\frac{\beta}{Q}}\frac{e^{\frac{\alpha}{Q}}}{e^{\frac{\alpha}{Q}}-1}+C(Q,\alpha,\beta)\right),\\
\end{aligned}
\end{equation}
where $C(Q,\alpha,\beta)=\sum_{i=0}^{\infty}\frac{i^{\frac{\beta}{Q}}}{e^{i\frac{\alpha}{Q}}}.$
Hence, we have proved that  $f$ satisfies inequality (\ref{90oiiiii}) on $A\cap B(x,R/4).$ The proof  that $f$ satisfies inequality (\ref{90oiiiii}) on $ B(x,R/4)\setminus A$ is the same as the one given in \cite{DKO24}. Hence we omit it.
\end{proof}
Based on the above preparations, we can give the $\mathbf{proof\,of\,Theorem}$ \ref{thm20}. Here we should deal with carefully because the mappings in Proposition \ref{09888888111} are real-value.
\begin{proof}
According to Lemma
\ref{lem0}, we see that there exist constants $\alpha>0$ and $\beta\geq0,$ such that the following inequality 
\begin{equation}\label{96giiuo}
\int_{B(x,r)}\frac{\abs{ D_{h}f}^{Q}}{K}\leq C r^{\alpha}\left(\log\frac{c_{1}R_{0}e^{1+I\mathcal{K}}}{r}\right)^{\beta}
\end{equation}
holds for almost every $x\in\Omega$ and $0<r<c_{1}R_{0}<{\rm dist}(x,\partial\Omega).$ Here, the constant 
 \begin{equation}\label{96giiuasaso}
  \alpha=\left\{
\begin{aligned}
&\min\Big\{1-\frac{1}{p}-\frac{1}{q},\frac{1}{QI\mathcal{K}}\Big\},       &      & {\rm if}\; {1-\frac{1}{p}-\frac{1}{q}\neq\frac{1}{QI\mathcal{K}}},\\
&\frac{1}{I\mathcal{K}},       &      &{\rm if}\; {1-\frac{1}{p}-\frac{1}{q}=\frac{1}{QI\mathcal{K}}},
\end{aligned} \right.
\end{equation}
 \begin{equation}\label{96giiuasasosdsscsccvr}
  \beta=\left\{
\begin{aligned}
&0,       &      & {\rm if}\; {1-\frac{1}{p}-\frac{1}{q}\neq\frac{1}{QI\mathcal{K}}},\\
&1,       &      &{\rm if}\; {1-\frac{1}{p}-\frac{1}{q}=\frac{1}{QI\mathcal{K}}}
\end{aligned} \right.
\end{equation}
and $C$ is a positive constant depends only on $Q,  p, q,\mathcal{K}$ and $\Vert D_{h}f\Vert_{L^{Q}(\Omega)}.$ Therefore, there exist constants $\alpha>0$ and $\beta\geq0,$ such that the following inequality 
\begin{equation}
\int_{B(x,r)}\frac{\abs{ \nabla_{h}\rho\circ f}^{Q}}{K}\leq\int_{B(x,r)}\frac{\abs{ D_{h}f}^{Q}}{K}\leq C r^{\alpha}\left(\log\frac{c_{1}R_{0}e^{1+I\mathcal{K}}}{r}\right)^{\beta}
\end{equation}
holds for almost every $x\in\Omega$ and $0<r<c_{1}R_{0}<{\rm dist}(x,\partial\Omega).$ In addition, it is easy to see that $\rho\circ f\in W_{loc}^{1,Q}(\Omega).$
Hence, according to Proposition \ref{09888888111}, we have $\rho\circ f\in\mathcal{C}(\Omega).$ Without loss of generality, we may 
assume that $\rho\circ f$ is bounded on $\Omega.$ Now, from the proof of \cite[Proposition 5]{Vod96}, we see that there exists a constant $C(\Omega),$ such that 
\begin{equation}\label{100000}
\vert\nabla_{h}f_{j}\vert\leq C(\Omega)\vert D_{h}f\vert
\end{equation} for all $j=1,2,\cdots,n,\cdots,N,$ where $N={\rm dim}\mathbb{G}.$ With inequality (\ref{100000}) and $\vert f_{j}\vert\preceq\rho\circ f$ or $\vert f_{j}\vert\preceq(\rho\circ f)^{2},$ we have
$f_{j}\in W_{loc}^{1,Q}(\Omega)$ all $j=1,2,\cdots,n,\cdots,N={\rm dim}\mathbb{G}.$
On the other hand, inequalities  (\ref{96giiuo})  and (\ref{100000}) give us that there exist constants $\alpha>0$ and $\beta\geq0,$ which given by (\ref{96giiuasaso}) and (\ref{96giiuasasosdsscsccvr}), such that the following inequality 
\begin{equation}
\int_{B(x,r)}\frac{\abs{ \nabla_{h}f_{j}}^{Q}}{K}\leq C(Q,\Omega)\int_{B(x,r)}\frac{\abs{ D_{h}f}^{Q}}{K}\leq C_{0} r^{\alpha}\left(\log\frac{c_{1}R_{0}e^{1+I\mathcal{K}}}{r}\right)^{\beta}
\end{equation}
holds for almost every $x\in\Omega$ and $0<r<c_{1}R_{0}<{\rm dist}(x,\partial\Omega)$ and for all $j=1,2,\cdots,n,\cdots,N.$ Here, $C_{0}=C(Q, \Omega)C.$ By Proposition \ref{09888888111}, the following inequality
\begin{equation}
\vert f_{j}(y)-f_{j}(z)\vert\leq C_{1}(d_{c}(z,y))^{\frac{\alpha}{Q}}\left(C_{2}\left(\log\frac{c_{1}R_{0}e^{1+I\mathcal{K}}}{d_{c}(z,y)}\right)^{\frac{\beta}{Q}}+C_{3}\right)
\end{equation}
holds for all $j=1,2,\cdots,n,\cdots,N$ and for every $y,z\in B(x,r/4)$ with $r$ is small enough. Without loss of generality, we may assume that $y=0=f(y).$ Then,  we get
\begin{equation}\label{556116}
\begin{aligned}
&d_{\rho}(f(z),f(y))=\rho\circ f(z)\\
=&\left( \sum_{j=1}^{n}\vert f_{j}(z)\vert^{4}+16\sum_{k=n+1}^{N}\vert f_{j}(z)\vert^{2}\right)^{1/4}\\
\leq&C_{1}(d_{c}(z,y))^{\frac{\alpha}{Q}}\left(C_{2}\left(\log\frac{c_{1}R_{0}e^{1+I\mathcal{K}}}{d_{c}(z,y)}\right)^{\frac{\beta}{Q}}+C_{3}\right)\\
&\;\;\;\;\;\;\;\;\;\;+C_{4}(d_{c}(z,y))^{\frac{\alpha}{2Q}}\left(C_{5}\left(\log\frac{c_{1}R_{0}e^{1+I\mathcal{K}}}{d_{c}(z,y)}\right)^{\frac{\beta}{2Q}}+C_{6}\right).\\
\end{aligned}
\end{equation}
This finishes the proof.
\end{proof}

According to Theorem \ref{thm20} and the Trudinger type inequality on Carnot group established by Saloff-Coste \cite{SC88}, we have the following corollary.
\begin{cor}\label{cor201}
Let $\varepsilon_{0}\in(0,1)$ be a given constant. Suppose that $\mathbb{G}$ is a Carnot group of $H$-type, $y_{0}\in\mathbb{G},$ $\Omega\subset\mathbb{G}$ is a domain and $f:\Omega\rightarrow \mathbb{G}$ is a mapping of class $W_{loc}^{1,Q}(\Omega)$ satisfying  the following inequality
$$\vert D_{h}f(x)\vert^{Q}\leq K(x)J_{f}(x)+\Sigma(x)(\rho\circ l_{y_{0}^{-1}}\circ f(x))^{Q}$$
for almost every $x\in\Omega.$ Let $K\in L_{loc}^{p}(\Omega)$ and $\Sigma\in L_{loc}^{1+\varepsilon}(\Omega),$ where $(1-\varepsilon_{0})(1+\varepsilon)/((1-\varepsilon_{0})\varepsilon-\varepsilon_{0})<p\leq\infty$ and $\varepsilon_{0}/(1-\varepsilon_{0})<\varepsilon\leq\infty.$  If we further assume that $K$  satisfies $(\mathcal{K},Q-1)$-spherical condition on $\Omega$, then $f$ has a continuous representative.
\end{cor}
\begin{proof}
Without loss of generality, we can only consider the case that $\Omega$ is a bounded domain, $y_{0}=0,$  $f\in W^{1,Q}(\Omega),$ $K\in L^{p}(\Omega)$ and $\Sigma\in L^{1+\varepsilon}(\Omega).$  Since $f\in W^{1,Q}(\Omega),$
we see that $\rho\circ f\in W^{1,Q}(\Omega).$ Next, we show that $\rho\circ f\in L_{loc}^{p_{1}}(\Omega)$ for any $p_{1}\geq1.$ For every compact subset $M$ with  $\overline{M}\subset\Omega,$ we take for every 
$x\in M.$  Donate $C_{M}={\rm dist}(\partial M,\partial\Omega)$ and consider  $\eta\in \mathcal{C}_{0}^{\infty}(\Omega),$ such that $\eta\in[0,1]$ on $\Omega,$  $\vert\nabla_{h} \eta\vert\leq 1/C_{M}$ on $B(x,C_{M}),$ $\eta=0$ on $\Omega\setminus B(x, C_{M})$ and $\eta=1$ on $B(x,C_{M}/2).$ By letting $g:=\eta\cdot\rho\circ f,$ then $g\in W_{0}^{1,Q}(\Omega).$ Now, by the Trudinger type inequality on Carnot group established by Saloff-Coste \cite{SC88}, we see that there exist constants  $A_{Q}>0 $ and $C_{0}>0,$ such that
$$\frac{1}{m(\Omega)}\int_{\Omega}\exp\left(A_{Q}\frac{\vert g\vert^{Q'}}{\|\nabla_{h}g\|_{Q}^{Q'}}\right)dx\leq C_{0},$$
where $Q'=Q/(Q-1).$  For any given $p_{1}\geq1,$ we take the minimum positive integer $m,$ such that it satisfies $m> p_{1}(Q-1)/Q.$ Then, by a basic inequality, which says that the inequality $\exp(a)\geq a^{n}/(n!)$ holds for every $a>0$ and every positive integer $n,$ we get
$$\frac{A_{Q}^{m}}{m!\cdot\|\nabla_{h}g\|_{Q}^{mQ'}}\int_{\Omega}\vert g(x)\vert^{\frac{Qm}{Q-1}}dx\leq\int_{\Omega}\exp\left(A_{Q}\frac{\vert g(x)\vert^{Q'}}{\|\nabla_{h}g(x)\|_{Q}^{Q'}}\right)dx\leq C(\Omega).$$
Since
$$\int_{\Omega}\vert\nabla_{h} g(x)\vert^{Q}dx\leq2^{Q-1}\left(\frac{1}{C_{M}^{Q}}\int_{\Omega}\vert\rho\circ f(x)\vert^{Q}dx+\int_{\Omega}\vert D_{h}f(x)\vert^{Q}dx\right),$$
we see that
$$\int_{B(x,C_{M}/2)}\vert \rho\circ f(x)\vert^{\frac{Qm}{Q-1}}dx=\int_{B(x,C_{M}/2)}\vert g(x)\vert^{\frac{Qm}{Q-1}}dx\leq\int_{\Omega}\vert g(x)\vert^{\frac{Qm}{Q-1}}dx\leq C,$$
where $C$ depends only on $M, m,Q,\Omega,\|\rho\circ f\|_{L^{Q}(\Omega)}$ and $\|D_{h} f\|_{L^{Q}(\Omega)}.$
Hence, by H\"{o}lder inequality and the boundedness of $\Omega,$ we get $\rho\circ f\in L_{loc}^{p_{1}}(\Omega)$ for any $p_{1}\geq1.$ 

Now,  we select $q$ such that $(1-\varepsilon_{0})<q/(1+\varepsilon)<1.$ In this situation, we have 
$$\frac{1}{p}+\frac{1}{q}<\frac{(1-\varepsilon_{0})\varepsilon-\varepsilon_{0}}{(1-\varepsilon_{0})(1+\varepsilon)}+\frac{1}{(1-\varepsilon_{0})(1+\varepsilon)}=1.$$ In addition, by using H\"{o}lder inequality, we derive for every compact subset $M$ with  $\overline{M}\subset\Omega$ that 
\begin{equation}
\begin{aligned}
&\int_{M}\left(\frac{\Sigma(\rho\circ f)^{Q}}{K}\right)^{q}dx\leq\int_{M}\Sigma^{q}(\rho\circ f)^{qQ}dx\\
\leq&\left(\int_{M}\Sigma^{1+\varepsilon}dx\right)^{\frac{q}{1+\varepsilon}}\cdot\left(\int_{M}(\rho\circ f)^{\frac{qQ(1+\varepsilon)}{1+\varepsilon-q}}dx\right)^{\frac{1+\varepsilon-q}{1+\varepsilon}}\\
\leq&\left(\int_{\Omega}\Sigma^{1+\varepsilon}dx\right)^{\frac{q}{1+\varepsilon}}\cdot\left(\int_{M}(\rho\circ f)^{\frac{qQ(1+\varepsilon)}{1+\varepsilon-q}}dx\right)^{\frac{1+\varepsilon-q}{1+\varepsilon}}<\infty.\\
\end{aligned}
\end{equation}
Hence, $\Sigma(\rho\circ f)^{Q}/K\in L_{loc}^{q}(\Omega).$ Now, the conclusion of Theorem \ref{thm20} yields  that $f$ has a continuous representative.
\end{proof}

In Corollary \ref{cor201}, if $\varepsilon=\infty,$ then $p$  can be taken any value that is greater or equals to 1. Meanwhile, if $p=\infty,$ then $\varepsilon$ can be taken as any positive number.   
In particular, if we take $K(x)\equiv K\geq1,$ then we can derive the following continuous result for quasiregular values on Carnot group of $H$-type.
\begin{cor}\label{cor20}
Let  $K\geq1$ be a given constant. Suppose that $\mathbb{G}$ is a  Carnot group of $H$-type, $\Omega\subset\mathbb{G}$ is a domain, $y_{0}\in\mathbb{G}$ and $f:\Omega\rightarrow \mathbb{G}$ is a mapping of class $W_{loc}^{1,Q}(\Omega)$ satisfying  the following inequality
$$\vert D_{h}f(x)\vert^{Q}\leq KJ_{f}(x)+\Sigma(x)(\rho\circ l_{y_{0}^{-1}}\circ f(x))^{Q}$$
for almost every $x\in\Omega.$ If  $\Sigma\in L_{loc}^{p}(\Omega)$ for some  $p>1,$
 then $f$ has a continuous representative.
\end{cor}

Next, we will end up this section by discussing two problems related to  Corollary \ref{corroo}. Recall a mapping $f\in W_{loc}^{1,Q}(\Omega,\mathbb{G})$ is said to be $K$-quasiconformal for some $K\geq1$ if it is homeomorphic and satisfies inequality (\ref{df2111111}) with $\Sigma\equiv0.$ Pansu \cite{Pan89}  has shown that 1-quasiconformal maps in Carnot groups are locally Lipschitz. Heinonen \cite{Hei95}  established the H\"{o}lder continuity for general $K$-quasiconformal maps in Carnot groups. However, the estimates he obtained is not asymptotically sharp. Namely, it does not suffice to give a H\"{o}lder exponent $\alpha=\alpha(K)$ which tends to one as $K\rightarrow1$. In additional, Balogh, Holopainen  and Tyson \cite{BHT02} showed a $K$-quasiconformal map between open subsets of a Carnot group $G$ is locally $K^{1/(1-Q)}$-H\"{o}lder continuous. If we consider the case that $1\leq K<(1+\varepsilon)/IQ\varepsilon$ in Corollary \ref{corroo}, then we obtain
  \begin{equation}
 \begin{aligned}
d_{c}(f(z),f(y))
\leq C(d_{c}(z,y))^{\frac{1}{2}\frac{\varepsilon}{1+\varepsilon}}\\
\end{aligned}
\end{equation}
for every fixed $x\in\Omega$ and every $y,z\in B(x,r)$ with $r$ is small enough. The H\"{o}lder exponent $\varepsilon/2(1+\varepsilon)\rightarrow1/2\neq1$ as $\varepsilon\rightarrow\infty.$
This leads to the following problem.
 \begin{prob}
 Find the asymptotically sharp H\"{o}lder exponent  in Corollary \ref{corroo}. That is to say: find a H\"{o}lder exponent $\theta(K,\varepsilon),$ such that 
 $$\lim_{K\rightarrow1^{+}, \varepsilon\rightarrow\infty}\theta(K,\varepsilon)=1.$$
 \end{prob}

On the other hand, in the case of Euclidean space, Kangasniemia and Onninen showed in \cite[Theorem 1.1]{KO221} that if a mapping $f\in W_{loc}^{1,n}(\Omega,\mathbb{R}^{n})$ satisfies
 $$\vert Df\vert^{n}\leq K J_{f}+\Sigma^{n}\vert f\vert^{n}$$
  for almost every $x\in\Omega$ and if $\Sigma\in L_{loc}^{1+\varepsilon}(\Omega)$ for some $\varepsilon>0.$ Then it is locally H\"{o}lder continuous. They further proved the following results:
  
   (i) If $\varepsilon/(n+\varepsilon)\neq1/K$, then the following H\"{o}lder exponent  $$\min\{\frac{\varepsilon}{n+\varepsilon},\frac{1}{K}\}$$
   is sharp.

    (ii)  If $\frac{\varepsilon}{n+\varepsilon}=\frac{1}{K}$, then it loaclly satisfies
$$\vert f(y)-f(z)\vert\leq C\vert y-z\vert^{1/K}\left(\log\frac{L}{\vert y-z\vert}\right)^{\frac{1}{n}}$$
and the  exponent  is sharp. 

In the case for quasiregular values on Carnot group of $H$-type, it is natural to ask the following problem.

 \begin{prob}
 Find the sharp H\"{o}lder exponent in Corollary \ref{corroo}.
 \end{prob}

\section{Discrete type result for quasiregular values}\label{sec3333}
In this section, we will study  the discreteness for quasiregular values on Carnot group of $H$-type. We sometimes provide some results on some subclass of generalized finite distortion maps.  The following result was obtained by  Basalaev and  Vodopyanov \cite[Lemma 2.4]{BV23}.
\begin{lem}{\rm\cite[Lemma 2.4]{BV23}}\label{l00}
Suppose that $f\in W^{1,Q}(\Omega)\cap \mathcal{C}(\Omega,\mathbb{G})$ and $\psi\in \mathcal{C}^{1}(\mathbb{R}^{+},\mathbb{R}^{+}),$ where $\psi(t)=o(t^{Q})$ as $t\rightarrow0.$ Then there exists a constant
$C=C(\mathbb{G})>0,$ such that the following inequality 
\begin{equation}\label{i0100}
\abs{\int_{\Omega}\varphi\frac{\psi'(\rho\circ f)\vert\nabla_{h}\rho\circ f\vert^{Q}}{(\rho\circ f)^{Q-1}}J_{f}}\leq C\abs{\int_{\Omega}\vert\nabla_{h}\varphi\vert\frac{\psi(\rho\circ f)\vert\nabla_{h}\rho\circ f\vert^{Q-1}}{(\rho\circ f)^{Q-1}}\vert D_{h}f\vert^{Q-1} }
\end{equation}
holds for every $\varphi\in \mathcal{C}_{0}^{\infty}(\Omega).$
\end{lem}
\begin{rem}
In Lemma \ref{l00}, the condition  $\psi(t)=o(t^{Q})$ as $t\rightarrow0$ is to ensure that the integration in (\ref{i0100}) is well-defined.  We can also chose those $\psi$ which is piecewise $C^{1}$-smooth with $\psi'$ locally bounded if the resulting integration in (\ref{i0100}) makes sense.
\end{rem}
By virtue of  Corollary \ref{cor201}  and Lemma \ref{l00}, we have the following result. 
\begin{thm}\label{th11000}
Let $\varepsilon_{0}\in(0,1)$ be a given constant. Suppose that $\mathbb{G}$ is a Carnot group of $H$-type, $\Omega\subset\mathbb{G}$ is a domain and $f:\Omega\rightarrow \mathbb{G}$ is a mapping of class $W_{loc}^{1,Q}(\Omega)$ satisfying  the following inequality
$$\vert D_{h}f(x)\vert^{Q}\leq K(x)J_{f}(x)+\Sigma(x)(\rho\circ f(x))^{Q}$$
for almost every $x\in\Omega.$ Let $K\in L_{loc}^{p}(\Omega)$ and $\Sigma\in L_{loc}^{1+\varepsilon}(\Omega),$ where $(1-\varepsilon_{0})(1+\varepsilon)/((1-\varepsilon_{0})\varepsilon-\varepsilon_{0})<p\leq\infty$ and $\varepsilon_{0}/(1-\varepsilon_{0})<\varepsilon\leq\infty.$  If we further assume that $K$ satisfies $(\mathcal{K},Q-1)$-spherical  condition on $\Omega,$ then $\nabla_{h}(\log^{+}\log^{+}\frac{1}{\rho}\circ f)\in L_{loc}^{q_{1}}(\Omega),$ where $q_{1}=pQ/(1+p)<Q.$ In particular, if $K(x)$ is a constant $K\geq1 ,$ then we have $\nabla_{h}(\log^{+}\log^{+}\frac{1}{\rho}\circ f)\in L_{loc}^{Q}(\Omega).$ 
\end{thm}
\begin{proof}
Without loss of generality, we may assume that $\Omega$ is a bounded domain, $f\in W^{1,Q}(\Omega),$ $K\in L^{p}(\Omega)$ and $\Sigma\in L^{1+\varepsilon}(\Omega).$
By Corollary \ref{cor201}, we see that $f\in \mathcal{C}(\Omega).$ Hence, it allows us to use  Lemma \ref{l00}. Since $\rho\circ f$ is continuous,  the set $$\Omega_{1}:=\{x\in\mathbb{G}:\rho\circ f(x)<1/e\}$$ is an open domain and $\log^{+}\log^{+}\frac{1}{\rho}\circ f(x)=0$ outside of $\Omega_{1}.$ So we may, without loss of generality, consider that $\Omega_{1}=\Omega.$ We fix $\varepsilon>0$ and consider the following mapping 
$$\psi_{\varepsilon}(t)=\int_{0}^{t}\frac{\varphi_{\varepsilon}(s)}{s\log^{Q}(\frac{1}{s})}ds,$$
where $$\varphi_{\varepsilon}(s)=\frac{1}{1+\varepsilon2^{\frac{1}{s}}}.$$
For every fixed $x_{0}\in \Omega,$ we assume that $0<r<2r<{\rm dist}(x_{0},\partial \Omega).$ Now, we are taking mappings $\psi=\psi_{\varepsilon}$ and $\varphi=\eta^{Q},$ where $\eta\in \mathcal{C}_{0}^{\infty}(\Omega)$ satisfying $\eta=1$ on $B(x_{0},r),$ $\vert\nabla_{h}\eta\vert\leq1/r$ on $B(x_{0},2r)$ and $\eta=0$ on $\Omega\setminus B(x_{0},2r).$
Then, by the following inequality 
$$\psi_{\varepsilon}(t)\leq\frac{\varphi_{\varepsilon}(t)^{\frac{Q-1}{Q}}}{(Q-1)\log^{Q-1}(\frac{1}{t})}$$
and Lemma \ref{l00}, we obtain 
\begin{equation}\label{09}
\begin{aligned}
&\int_{\Omega}\eta^{Q}\frac{\vert\nabla_{h}\rho\circ f\vert^{Q}J_{f}}{(\rho\circ f)^{Q}\log^{Q}(\frac{1}{\rho\circ f})}\varphi_{\varepsilon}(\rho\circ f)\\
\leq& C_{1} \int_{\Omega}\eta^{Q-1}\vert\nabla_{h}\eta\vert\frac{\vert\nabla_{h}\rho\circ f\vert^{Q-1}\vert D_{h}f\vert^{Q-1}}{(\rho\circ f)^{Q-1}}\psi_{\varepsilon}(\rho\circ f) \\
\leq& C_{2} \int_{\Omega}\eta^{Q-1}\vert\nabla_{h}\eta\vert\frac{\vert\nabla_{h}\rho\circ f\vert^{Q-1}\vert D_{h}f\vert^{Q-1}}{(\rho\circ f)^{Q-1}\log^{Q-1}(\frac{1}{\rho\circ f})}\left(\varphi_{\varepsilon}(\rho\circ f)\right)^{\frac{Q-1}{Q}}\\
\leq& C_{2} \left[\int_{\Omega}\eta^{Q}\frac{\vert\nabla_{h}\rho\circ f\vert^{Q}\vert D_{h}f\vert^{Q}}{K(\rho\circ f)^{Q}\log^{Q}(\frac{1}{\rho\circ f})}\varphi_{\varepsilon}(\rho\circ f)\right]^{\frac{Q-1}{Q}}\cdot\left[\int_{\Omega}K^{Q-1}\vert\nabla_{h}\eta\vert^{Q}\right]^{\frac{1}{Q}}.\\
\end{aligned}
\end{equation}
Hence,  by combining $$\frac{\vert D_{h}f\vert^{Q}}{K}-\frac{\Sigma(\rho\circ f)^{Q}}{K}\leq J_{f}$$ and inequality (\ref{09}), we get that 
\begin{equation}\label{07}
\begin{aligned}
&\int_{\Omega}\eta^{Q}\frac{\vert\nabla_{h}\rho\circ f\vert^{Q}\vert D_{h}f\vert^{Q}}{K(\rho\circ f)^{Q}\log^{Q}(\frac{1}{\rho\circ f})}\varphi_{\varepsilon}(\rho\circ f)\\
\leq& C_{2} \left[\int_{\Omega}\eta^{Q}\frac{\vert\nabla_{h}\rho\circ f\vert^{Q}\vert D_{h}f\vert^{Q}}{K(\rho\circ f)^{Q}\log^{Q}(\frac{1}{\rho\circ f})}\varphi_{\varepsilon}(\rho\circ f)dx\right]^{\frac{Q-1}{Q}}\cdot\left[\int_{\Omega}K^{Q-1}\vert\nabla_{h}\eta\vert^{Q}\right]^{\frac{1}{Q}}\\
&\;\;\;\;\;\;\;\;\;\;\;\;\;\;\;\;\;\;\;\;\;\;\;\;\;\;\;\;\;\;\;\;\;\;+\int_{\Omega}\eta^{Q}\frac{\vert\nabla_{h}\rho\circ f\vert^{Q}}{\log^{Q}(\frac{1}{\rho\circ f})}\varphi_{\varepsilon}(\rho\circ f)\frac{\Sigma}{K}\\
\leq& C_{3} \left[\int_{\Omega}\eta^{Q}\frac{\vert\nabla_{h}\rho\circ f\vert^{Q}\vert D_{h}f\vert^{Q}}{K(\rho\circ f)^{Q}\log^{Q}(\frac{1}{\rho\circ f})}\varphi_{\varepsilon}(\rho\circ f)\right]^{\frac{Q-1}{Q}}\cdot\left[\frac{\int_{B(x_{0},2r)}K^{Q-1}}{(2r)^{Q}}\right]^{\frac{1}{Q}}\\
&\;\;\;\;\;\;\;\;\;\;\;\;\;\;\;\;\;\;\;\;\;\;\;\;\;\;\;\;\;\;\;\;\;\;+\int_{\Omega}\eta^{Q}\frac{\vert\nabla_{h}\rho\circ f\vert^{Q}}{\log^{Q}(\frac{1}{\rho\circ f})}\varphi_{\varepsilon}(\rho\circ f)\frac{\Sigma}{K}\\
\leq& C_{4} \left[\int_{\Omega}\eta^{Q}\frac{\vert\nabla_{h}\rho\circ f\vert^{Q}\vert Df\vert^{Q}}{K(\rho\circ f)^{Q}\log^{Q}(\frac{1}{\rho\circ f})}\varphi_{\varepsilon}(\rho\circ f)\right]^{\frac{Q-1}{Q}}\\
&\;\;\;\;\;\;\;\;\;\;\;\;\;\;\;\;\;\;\;\;\;\;\;\;\;\;\;\;\;\;\;\;\;\;+\int_{\Omega}\eta^{Q}\frac{\vert\nabla_{h}\rho\circ f\vert^{Q}}{\log^{Q}(\frac{1}{\rho\circ f})}\varphi_{\varepsilon}(\rho\circ f)\frac{\Sigma}{K}.\\
\end{aligned}
\end{equation}

Next, we will finish the proof of Theorem \ref{th11000} in the following two steps.

${\rm \mathbf{Step}\;\mathbf{1}}$: To prove that  $\vert \nabla_{h}\rho\circ f\vert\leq 1$ and hence derive the finiteness of the integration
\begin{equation}\label{000009}
\int_{\Omega}\eta^{Q}\frac{\vert\nabla_{h}\rho\circ f\vert^{Q}}{\log^{Q}(\frac{1}{\rho\circ f})}\frac{\Sigma}{K}.
\end{equation}

Actually,  this can be seen by letting $v^{1/4}=\rho=\left(\vert a\vert^{4}+16\vert b\vert^{2}\right)^{1/4}$ and $\Phi_{j}(t)=v(x\exp tX_{j}),j=1,\cdots,n.$ Then $X_{j}v=\Phi'_{j}(0).$ Hence, by \cite[(15)]{Kap80}, we get
$$16\rho^{6}\sum_{j=1}^{n}(X_{j}\rho)^{2}=\sum_{j=1}^{n}(X_{j}v)^{2}=16v\vert a\vert^{2}\leq16 v^{3/2}=16\rho^{6}.$$
Hence, we get $\vert \nabla_{h}\rho\circ f\vert\leq 1.$ This result also can be derived directly by the following formula \cite[Proposition 3.2.16]{Ing10}
\begin{equation}\label{0707}
\vert\nabla_{h}\rho(x)\vert=\frac{\vert a(x)\vert}{\rho(x)},
\end{equation}
where $x=\exp\left(a(x)+b(x)\right).$   Since $\rho\circ f<1/e,$ we see that
\begin{equation}\label{05}
\begin{aligned}
&\int_{\Omega}\eta^{Q}(x)\frac{\vert\nabla_{h}\rho\circ f(x)\vert^{Q}}{\log^{Q}(\frac{1}{\rho\circ f(x)})}\frac{\Sigma}{K}dx\\
\leq& C_{4}\int_{\Omega}\frac{\Sigma}{K}dx\leq\left(\int_{\Omega}\left(\frac{\Sigma}{K}\right)^{q}dx\right)^{\frac{1}{q}}\left(\int_{\Omega}1^{\frac{q}{q-1}}dx\right)^{\frac{q-1}{q}}<\infty. \\
\end{aligned}
\end{equation}
This finishes the finiteness of integration (\ref{000009}).

${\rm \mathbf{Step}\;\mathbf{2}}$:  To verify the following inequality
\begin{equation}\label{0000018}
\int_{\Omega}\eta^{Q}\frac{\vert\nabla_{h}\rho\circ f(x)\vert^{Q}\vert D_{h}f(x)\vert^{Q}}{K(x)(\rho\circ f(x))^{Q}\log^{Q}(\frac{1}{\rho\circ f(x)})}dx\leq C,
\end{equation}
where $C$ is a constant depends only on $Q,\mathcal{K},\|K\|_{L^{p}},\|\Sigma/K\|_{L^{q}}$ and $\Omega.$ To show this conclusion, we combine inequalities (\ref{07}) and (\ref{05}) and obtain
\begin{equation}\label{06}
\begin{aligned}
&\int_{\Omega}\eta^{Q}(x)\frac{\vert\nabla_{h}\rho\circ f(x)\vert^{Q}\vert D_{h}f(x)\vert^{Q}}{K(\rho\circ f(x))^{Q}\log^{Q}(\frac{1}{\rho\circ f(x)})}\varphi_{\varepsilon}(\rho\circ f)dx\\
\leq& C_{3} \left[\int_{\Omega}\eta^{Q}(x)\frac{\vert\nabla_{h}\rho\circ f(x)\vert^{Q}\vert D_{h}f(x)\vert^{Q}}{K(\rho\circ f(x))^{Q}\log^{Q}(\frac{1}{\rho\circ f(x)})}\varphi_{\varepsilon}(\rho\circ f(x))dx\right]^{\frac{Q-1}{Q}}+C_{5}.\\
\end{aligned}
\end{equation}
Now, if there exists constant $C_{6},$ such that
$$\int_{\Omega}\eta^{Q}(x)\frac{\vert\nabla_{h}\rho\circ f(x)\vert^{Q}\vert D_{h}f(x)\vert^{Q}}{K(\rho\circ f(x))^{Q}\log^{Q}(\frac{1}{\rho\circ f(x)})}\varphi_{\varepsilon}(\rho\circ f(x))dx\geq C_{6}.$$
Then, inequality (\ref{06}) gives that 
$$
\left(\int_{\Omega}\eta^{Q}\frac{\vert\nabla_{h}\rho\circ f\vert^{Q}\vert D_{h}f\vert^{Q}}{K(\rho\circ f)^{Q}\log^{Q}(\frac{1}{\rho\circ f})}\varphi_{\varepsilon}(\rho\circ f)\right)^{\frac{1}{Q}}\leq
C_{3}+C_{5}/(C_{6})^{(Q-1)/Q},
$$
which ensures the  finiteness of the integration (\ref{0000018}) according to the monotone convergence Theorem.

Now, we can give the proof of  $\nabla_{h}(\log^{+}\log^{+}\frac{1}{\rho}\circ f)\in L_{loc}^{q_{1}}(\Omega).$ To do this,we take every $x_{0}\in\Omega,$ $0<r<{\rm dist}(x_{0},\partial\Omega)$ and $q_{1}=pQ/(1+p).$  Then, by using of H\"{o}lder inequality and
inequality (\ref{0000018}), we get
\begin{equation}\label{04}
\begin{aligned}
&\int_{B(x_{0},r)}\abs{\nabla_{h}(\log^{+}\log^{+}\frac{1}{\rho}\circ f(x))}^{q_{1}}dx\\
\leq&\int_{B(x_{0},r)}\frac{\vert\nabla_{h}\rho\circ f(x)\vert^{q_{1}}\vert D_{h}f(x)\vert^{q_{1}}}{(\rho\circ f(x))^{q_{1}}\log^{q_{1}}(\frac{1}{\rho\circ f(x)})}dx\\
=&\int_{B(x_{0},r)}\left[\frac{\vert\nabla_{h}\rho\circ f(x)\vert^{Q}\vert D_{h}f(x)\vert^{Q}}{K(\rho\circ f(x))^{Q}\log^{Q}(\frac{1}{\rho\circ f(x)})}\right]^{\frac{q_{1}}{Q}}\cdot K^{\frac{q_{1}}{Q}}(x)dx\\
\leq& \left[\int_{B(x_{0},r)}\frac{\vert\nabla_{h}\rho\circ f(x)\vert^{Q}\vert D_{h}f(x)\vert^{Q}}{K(\rho\circ f(x))^{Q}\log^{Q}(\frac{1}{\rho\circ f(x)})}dx\right]^{\frac{q_{1}}{Q}}\cdot \left[\int_{B(x_{0},r)}K^{p}(x)dx\right]^{\frac{1}{p}}\\
\leq& \left[\int_{\Omega}\eta^{Q}\frac{\vert\nabla_{h}\rho\circ f(x)\vert^{Q}\vert D_{h}f(x)\vert^{Q}}{K(\rho\circ f(x))^{Q}\log^{Q}(\frac{1}{\rho\circ f(x)})}dx\right]^{\frac{q_{1}}{Q}}\cdot \left[\int_{\Omega}K^{p}(x)dx\right]^{\frac{1}{p}}<\infty.\\
\end{aligned}
\end{equation}
This shows that $\nabla_{h}(\log^{+}\log^{+}\frac{1}{\rho}\circ f)\in L_{loc}^{q_{1}}(\Omega).$ In addition, the conclusion of $\nabla_{h}(\log^{+}\log^{+}\frac{1}{\rho}\circ f)\in L_{loc}^{Q}(\Omega)$ can be derived directly by inequality (\ref{0000018}). This finishes the proof of Theorem \ref{th11000}.
\end{proof}
By above result, we have the following lemma.
\begin{lem}\label{101209}
Let $\varepsilon_{0}\in(0,1)$ be a given constant. Suppose that $\mathbb{G}$ is a Carnot group of $H$-type,  $y_{0}\in\mathbb{G},$ $\Omega\subset\mathbb{G}$ is a domain and $f:\Omega\rightarrow \mathbb{G}$ is a mapping of class $W_{loc}^{1,Q}(\Omega)$ satisfying  the following inequality
$$\vert D_{h}f(x)\vert^{Q}\leq K(x)J_{f}(x)+\Sigma(x)(\rho\circ l_{y_{0}^{-1}}\circ f(x))^{Q}$$
for almost every $x\in\Omega.$ Let $K\in L_{loc}^{p}(\Omega)$ and $\Sigma\in L_{loc}^{1+\varepsilon}(\Omega),$ where $(1-\varepsilon_{0})(1+\varepsilon)/((1-\varepsilon_{0})\varepsilon-\varepsilon_{0})<p\leq\infty$ and $\varepsilon_{0}/(1-\varepsilon_{0})<\varepsilon\leq\infty.$  If we further assume that $K$  satisfies $(\mathcal{K},Q-1)$-spherical  condition on $\Omega,$ then either $f\equiv y_{0},$ or $\mathcal{H}^{1}(f^{-1}\{y_{0}\})=0.$  In the latter case, $f^{-1}\{y_{0}\}$ is totally disconnected.
\end{lem}
\begin{proof}
Without loss of generality, we here only consider the case for $y_{0}=0.$ We may chose a neighborhood of $y_{0}=0$  as $V:=\{y\in\mathbb{G}:d_{\rho}(y,0)<1/e\},$ such that for every ball $B$ compactly contained in $f^{-1}V,$  we have either $f\equiv 0$ on $B$ or $\log\log\frac{1}{\rho}\circ f\in W^{1,q_{1}}(B),$ and they cannot occur simultaneously. Here  $q_{1}=pQ/(1+p)<Q.$ Now, if $f\equiv 0$ on $B,$ then the continuity, which  established in Theorem \ref{thm20},
implies that $B=\Omega.$ Hence, we next prove the later case.

 Let $u_{k}:=\min\{\log\log\frac{1}{\rho}\circ  f,k\}$ for every positive integer $k.$ Since $\rho\circ f<1/e,$ then $u_{k}\geq0$ and $u_{k}\in W^{1,Q}(B).$  Let $g(x):=\log\log\frac{1}{\rho}\circ  f(x), x\in B.$ We assume that $g$ is not identically $\infty$ on $B.$ Then, there exists $0<t_{0}<1,$ such that the set $A:=\{x\in B:t_{0}<g(x)<1/t_{0}\}$ has positive  measure. With this condition, we have $g\in L^{1}(B).$ This can be seem by using of  Poincar\'{e} inequality in Lemma \ref{-11} and Theorem \ref{th11000}. We give the details of proof here. Firstly, by H\"{o}lder inequality, Poincar\'{e} inequality  (\ref{0311888888000}) and Theorem \ref{th11000}, we see that 
 \begin{equation}
\begin{aligned}
&\int_{B}\vert u_{k}-(u_{k})_{B}\vert dx\\
\leq& C(B,q_{1})\left(\int_{B}\vert u_{k}-(u_{k})_{B}\vert^{q_{1}} dx\right)^{\frac{1}{q_{1}}}\leq C(B,q_{1})\left(\int_{B}\vert \nabla_{h}u_{k}\vert^{q_{1}} dx\right)^{\frac{1}{q_{1}}}\\
\leq&C(B,q_{1},\mathcal{K})\left(\int_{B}\frac{\vert\nabla_{h}\rho\circ f(x)\vert^{Q}\vert D_{h}f(x)\vert^{Q}}{K(x)(\rho\circ f(x))^{Q}\log^{Q}(\frac{1}{\rho\circ f(x)})}dx\right)^{\frac{1}{Q}}\leq C,\\
\end{aligned}
\end{equation}
 where $C$ does not depend on $k.$ In addition, 
 \begin{equation}
\begin{aligned}
&\int_{B}\vert u_{k}-(u_{k})_{B}\vert dx\geq\int_{A}((u_{k})_{B}-u_{k} )dx\\
\geq&\int_{A}((u_{k})_{B}-g )dx\geq\int_{A}((u_{k})_{B}-1/t_{0} )dx.
\end{aligned}
\end{equation}
  Hence,
  $$\int_{B}u_{k}dx\leq C(m(A),B,t_{0},q_{1}, Q)<\infty.$$
 By monotone convergence Theorem, we get that $g\in L^{1}(B).$ Next, we will show that $g\in  L^{q_{1}}(B).$ By Poincar\'{e} inequality   (\ref{0311888888000})  and Theorem \ref{th11000}, we get 
\begin{equation}
\begin{aligned}
&\int_{B}\vert u_{k}\vert^{q_{1}}dx\\
\leq& 2^{q_{1}-1}\int_{B}\vert u_{k}-(u_{k})_{B}\vert^{q_{1}}dx+2^{q_{1}-1}\int_{B}\vert(u_{k})_{B}\vert^{q_{1}}dx\\
\leq& C_{1}(B,q_{1})\int_{B}\vert \nabla_{h}u_{k}\vert^{q_{1}}dx+C_{2}(B,q_{1})\left(\int_{B} g dx\right)^{q_{1}}\\
\leq& C(B,q_{1},\|K\|_{L^{p}(\overline{B})})\left(\int_{B}\frac{\vert\nabla_{h}\rho\circ f(x)\vert^{Q}\vert D_{h}f(x)\vert^{Q}}{K(x)(\rho\circ f(x))^{Q}\log^{Q}(\frac{1}{\rho\circ f(x)})}dx\right)^{\frac{p}{1+p}}\\
\,\,\,\,\,\,\,\,\,\,\,\,\,\,\,\,\,\,\,\,\,\,\,\,+&C_{2}(B,q_{1})\left(\int_{B} g dx\right)^{q_{1}}\leq C,\\
\end{aligned}
\end{equation}
where $C$ does not depend  on $k.$ Hence, by  monotone convergence Theorem again, we get that $g\in L^{q_{1}}(B).$ In addition, Theorem \ref{th11000} and monotone convergence Theorem  make sure that 
$\nabla_{h}g\in L^{q_{1}}(B).$ Therefore, we have finished the proof of $\log\log\frac{1}{\rho}\circ f\in W^{1,q_{1}}(B),$ where $q_{1}=pQ/(1+p).$

Let $x\in f^{-1}\{0\}$ and $B=B(x,r),$ such that  $\overline{3B}=\overline{B(x,3r)}\subset f^{-1}V.$ Taking $\eta\in \mathcal{C}_{0}^{\infty}(3B),$ such that $0\leq\eta\leq1$ and $\eta=1$ on $2B.$
Define $w_{k}:=\eta u_{k}/k,$ where $u_{k}=\min\{\log\log\frac{1}{\rho}\circ f,k\}.$ Then, $w_{k}\in W_{0}^{1,q_{1}}(3B)$ and $w_{k}\equiv1$ in a neighborhood of $f^{-1}\{0\}\cap \overline{B}$ for every positive integer $k.$ Therefore, $w_{k}$ are admissible for the $q_{1}$-capacity of the condenser
$(f^{-1}\{0\}\cap \overline{B},3B).$ In addition, since
$$\int_{B}\vert w_{k}(x)\vert^{q_{1}}dx\preceq \frac{\int_{B}\vert\log\log\frac{1}{\rho}\circ f(x)\vert^{q_{1}}dx}{k}\preceq\frac{C_{1}}{k}\rightarrow0\;\;{\rm as}\;k\rightarrow\infty$$
and 
$$\int_{B}\vert \nabla_{h}w_{k}(x)\vert^{q_{1}}dx\preceq \frac{\left(\int_{B}\frac{\vert\nabla_{h}\rho\circ f(x)\vert^{Q}\vert D_{h}f(x)\vert^{Q}}{K(x)(\rho\circ f(x))^{Q}\log^{Q}(\frac{1}{\rho\circ f(x)})}dx\right)^{\frac{p}{1+p}}}{k}\preceq \frac{C}{k}\rightarrow0\;\;{\rm as}\;k\rightarrow\infty,$$
we see that ${\rm Cap}_{q_{1}}(f^{-1}\{0\}\cap \overline{B},3B)=0.$ Hence, by \cite[Theorem 9]{VK09}, we get $\mathcal{H}^{1}(f^{-1}\{0\}\cap \overline{B})=0.$ The claim $\mathcal{H}^{1}(f^{-1}\{0\})=0$
 follows by considering a countable cover of $f^{-1}\{0\}$ by such $B.$ This finishes the  proof of Lemma \ref{101209}.
\end{proof}

Thanks to the totally disconnected which have been established in Lemma \ref{101209}. The Lemma 3.1 in  \cite{KO221} also works on the setting of Carnot group of $H$-type. The readers can directly prove the following Lemma \ref{9981} by following the procedure of \cite[Lemma 3.1]{KO221}. Hence, we omit it here.
\begin{lem}\label{9981}
Let $\varepsilon_{0}\in(0,1)$ be a given constant. Suppose that $\mathbb{G}$ is a Carnot group of $H$-type,  $y_{0}\in\mathbb{G},$ $\Omega\subset\mathbb{G}$ is a domain and $f:\Omega\rightarrow \mathbb{G}$ is a mapping of class $W_{loc}^{1,Q}(\Omega)$ satisfying  the following inequality
$$\vert D_{h}f(x)\vert^{Q}\leq K(x)J_{f}(x)+\Sigma(x)(\rho\circ l_{y_{0}^{-1}}\circ f(x))^{Q}$$
for almost every $x\in\Omega.$ Let $K\in L_{loc}^{p}(\Omega)$ and $\Sigma\in L_{loc}^{1+\varepsilon}(\Omega),$ where $(1-\varepsilon_{0})(1+\varepsilon)/((1-\varepsilon_{0})\varepsilon-\varepsilon_{0})<p\leq\infty$ and $\varepsilon_{0}/(1-\varepsilon_{0})<\varepsilon\leq\infty.$  If we further assume that $K$  satisfies $(\mathcal{K},Q-1)$-spherical  condition on $\Omega.$ Then, there exists $\varepsilon_{1}>0$ such that if $U_{\varepsilon_{1}}$ is the $x$-component of $f^{-1}B(y_{0},\varepsilon_{1}),$ then $\overline{U_{\varepsilon_{1}}}$ is compactly contained in $\Omega.$
\end{lem}

A Jacobian formula expresses by the degree of a Sobolev map $f$ also holds in Carnot group of $H$-type. 
\begin{lem}\label{lomlm}
For a given $\varepsilon_{0}\in(0,1),$ we suppose that $\varepsilon_{0}/(1-\varepsilon_{0})<\varepsilon\leq\infty$ and $(1-\varepsilon_{0})(1+\varepsilon)/((1-\varepsilon_{0})\varepsilon-\varepsilon_{0})<p\leq\infty.$ Let $\mathbb{G}$ be a Carnot group of $H$-type,  $y_{0}\in\mathbb{G},$ $\Omega\subset\mathbb{G}$ be a domain and $f:\Omega\rightarrow \mathbb{G}$ be an nonconstant mapping of class $W_{loc}^{1,Q}(\Omega)$ satisfying  the following inequality
$$\vert D_{h}f(x)\vert^{Q}\leq K(x)J_{f}(x)+\Sigma(x)(\rho\circ l_{y_{0}^{-1}}\circ f(x))^{Q}$$
for almost every $x\in\Omega.$ Suppose that $K\in L_{loc}^{p}(\Omega),$  $\Sigma\in L_{loc}^{1+\varepsilon}(\Omega)$  and $K$ satisfies $(\mathcal{K},Q-1)$-spherical  condition on $\Omega.$ Let $U_{\varepsilon_{1}}$ be a connected component of $f^{-1}B(y_{0},\varepsilon_{1}),$ which $\overline{U_{\varepsilon_{1}}}$ is compactly contained in $\Omega.$ Then, we have
\begin{equation}
{\rm deg}(f,U_{\varepsilon_{1}})=\frac{1}{\varepsilon_{1}^{Q}}\int_{U_{\varepsilon_{1}}}J_{f}.
\end{equation}
\end{lem}

We can also work  Lemma 3.5  in  \cite{KO221}  on the setting of  Carnot group of $H$-type. It allows us to make the Haar measure of a component of $f^{-1}B(y_{0},\varepsilon_{1})$ small enough as $\varepsilon_{1}$ tends to zero. The proof of Lemma 3.5  in  \cite{KO221}  use the fact that  $f^{-1}\{y_{0}\}$ is totally disconnected. Here we offer another method for proving this result on Carnot group of $H$-type.
\begin{lem}
 For a given $\varepsilon_{0}\in(0,1),$ we suppose that $\varepsilon_{0}/(1-\varepsilon_{0})<\varepsilon\leq\infty$ and $(1-\varepsilon_{0})(1+\varepsilon)/((1-\varepsilon_{0})\varepsilon-\varepsilon_{0})<p\leq\infty.$ Let $\mathbb{G}$ be a Carnot group of $H$-type,  $y_{0}\in\mathbb{G},$ $\Omega\subset\mathbb{G}$ be a domain and $f:\Omega\rightarrow \mathbb{G}$ be an nonconstant mapping of class $W_{loc}^{1,Q}(\Omega)$ satisfying  the following inequality
$$\vert D_{h}f(x)\vert^{Q}\leq K(x)J_{f}(x)+\Sigma(x)(\rho\circ l_{y_{0}^{-1}}\circ f(x))^{Q}$$
for almost every $x\in\Omega.$ Suppose that $K\in L_{loc}^{p}(\Omega),$  $\Sigma\in L_{loc}^{1+\varepsilon}(\Omega)$  and $K$ satisfies $(\mathcal{K},Q-1)$-spherical  condition on $\Omega.$ For all $\varepsilon_{1}>0,$ let $x_{0}\in f^{-1}\{y_{0}\}$ and $U_{\varepsilon_{1}}$ be the $x_{0}$-component of $f^{-1}B(y_{0},\varepsilon_{1}).$ Then
$\lim_{\varepsilon_{1}\rightarrow0^{+}}m(U_{\varepsilon_{1}})=0.$
\end{lem}
\begin{proof}
According to Lemma \ref{9981}, we see that there exists $\varepsilon_{1}>0$ such that  $U_{\varepsilon_{1}}$ is the $x_{0}$-component of $f^{-1}B(y_{0},\varepsilon_{1})$ and $\overline{U_{\varepsilon_{1}}}$ is compactly contained in $\Omega.$ Now, for all those $0<\varepsilon<\varepsilon_{1},$ we see that $m(U_{\varepsilon})\leq m(f^{-1}B(y_{0},\varepsilon)).$ Since both $U_{\varepsilon}$ and $f^{-1}B(y_{0},\varepsilon)$ are all decreasing as $\varepsilon$ decreases, we get 
$$\lim_{\varepsilon\rightarrow0^{+}}m(U_{\varepsilon})\leq \lim_{\varepsilon\rightarrow0^{+}}m(f^{-1}B(y_{0},\varepsilon)).$$
Hence, we have $\lim_{\varepsilon\rightarrow0^{+}}m(U_{\varepsilon})=0$ if $m(f^{-1}\{y_{0}\})=0.$ The later one can be derived from the local integrability for real value $g=\log\log\frac{1}{\rho}\circ l_{y_{0}^{-1}}\circ f: B\rightarrow \mathbb{R};$ see from proof the of Lemma \ref{101209} on this local integrability.
\end{proof}
The following lemma is analogue to \cite[Lemma 4.1]{KO221} in the setting of Carnot group of $H$-type. We leave the proof for the readers.
\begin{lem} \label{987887}
 For a given $\varepsilon_{0}\in(0,1),$ we suppose that $\varepsilon_{0}/(1-\varepsilon_{0})<\varepsilon\leq\infty$ and $(1-\varepsilon_{0})(1+\varepsilon)/((1-\varepsilon_{0})\varepsilon-\varepsilon_{0})<p\leq\infty.$ Let $\mathbb{G}$ be a Carnot group of $H$-type,  $y_{0}\in\mathbb{G},$ $\Omega\subset\mathbb{G}$ be a domain and $f:\Omega\rightarrow \mathbb{G}$ be an nonconstant mapping of class $W_{loc}^{1,Q}(\Omega)$ satisfying  the following inequality
$$\vert D_{h}f(x)\vert^{Q}\leq K(x)J_{f}(x)+\Sigma(x)(\rho\circ l_{y_{0}^{-1}}\circ f(x))^{Q}$$
for almost every $x\in\Omega.$ Suppose that $K\in L_{loc}^{p}(\Omega),$  $\Sigma\in L_{loc}^{1+\varepsilon}(\Omega)$  and $K$ satisfies $(\mathcal{K},Q-1)$-spherical  condition on $\Omega.$ Suppose that $U$ is a non-empty component of $f^{-1}B(y_{0},\varepsilon_{1})$ for $\varepsilon_{1}>0$ is small enough such that $\overline{U}\subset\Omega.$ Then there is a constant $C=C(Q,K,\Sigma,\Omega)>0,$ such that ${\rm deg}\;(f,U)\geq0$ if $m(U)<C.$
\end{lem}
Now, we will establish the following discreteness for mappings with quasiregular values on Carnot group of $H$-type.
\begin{thm}\label{theorem9009}
Let $K\geq1$ be a given constant. Suppose that $\mathbb{G}$ is a Carnot group of $H$-type,  $y_{0}\in\mathbb{G},$ $\Omega\subset\mathbb{G}$ is a domain and $f:\Omega\rightarrow \mathbb{G}$ is an nonconstant mapping of class $W_{loc}^{1,Q}(\Omega)$ satisfying  the following inequality
$$\vert D_{h}f(x)\vert^{Q}\leq KJ_{f}(x)+\Sigma(x)(\rho\circ l_{y_{0}^{-1}}\circ f(x))^{Q}$$
for almost every $x\in\Omega.$ Suppose that $\Sigma\in L_{loc}^{p}(\Omega)$ for some $p>1,$ then $f^{-1}\{y_{0}\}$ is discrete.
\end{thm}
\begin{proof}
We should prove that for every two different points $x_{1,},x_{2}\in f^{-1}\{y_{0}\},$ there exist two neighborhoods $U_{1}$ of $x_{1}$ and $U_{2}$ of $x_{2},$ such that $U_{1}\cap U_{2}=\emptyset.$
We prove this by contradiction. Then, there exists a point $x_{0}\in f^{-1}\{y_{0}\},$ such that for every neighborhood $U$ of $x_{0},$ there always exists a point $x\in f^{-1}\{y_{0}\}\setminus\{x_{0}\},$ such that $x\in U.$

By Lemmas \ref{9981} and \ref{987887}, we see that it can chose a $\varepsilon_{0}>0,$ such that there is an non-empty $x_{0}$-component $U_{0}$ of $f^{-1}B(y_{0},\varepsilon_{0})$  satisfying the following conclusions:
 
 $(1)\;$$\overline{U_{0}}\subset\Omega;$
 
 $(2)\;$there exists a constant $C=C(Q,\mathcal{K},\Sigma,\Omega)>0$ with $m(U_{0})<C,$ such that  ${\rm deg}\;(f,U_{0})\geq0.$

Since there must exist 
$x_{1}\in U_{0}\setminus\{x_{0}\}$ and $ x_{1}\in f^{-1}\{y_{0}\}\setminus\{x_{0}\}.$ Hence, by the totally disconnectedness of $f^{-1}\{y_{0}\}$ from Lemma \ref{101209}, we can chose a $0<\varepsilon_{1}<\varepsilon_{0},$ such that $U_{1}\subset U_{0}$ is an neighborhood of $x_{0}$ satisfying that 
$x_{1}\notin U_{1}$ and $m(U_{1})\leq m(U_{0})< C.$ We let $U_{1,j}$ be the other components of $f^{-1}B(y_{0},\varepsilon_{1}),$ then there must exist a components $U_{1,j_{0}}$ of $x_{1},$ such that 
$m(U_{1,j_{0}})\leq m(U_{0})<C.$ Hence, by Lemma \ref{lemma000829992}, we have ${\rm deg}(f,U_{1,j_{0}})>0,$ which implies that 
$${\rm deg}(f,U_{0})={\rm deg}(f,U_{1})+{\rm deg}(f,U_{1,j_{0}})+\sum_{j\neq j_{0}}{\rm deg}(f,U_{1,j})>{\rm deg}(f,U_{1}).$$ 
By the same procedure, we get that there exist $U_{2},U_{3},\cdots,$ such that 
$${\rm deg}(f,U_{0})>{\rm deg}(f,U_{1})>{\rm deg}(f,U_{2})>{\rm deg}(f,U_{3})>\cdots.$$
This is impossible since  ${\rm deg}(f,U_{0})$  is not only a positive integer, but also has a upper bound controlled by the following integral value
$${\rm deg}(f,U_{0})=\frac{1}{\varepsilon_{0}^{Q}}\int_{U_{0}}J_{f}\leq\frac{1}{\varepsilon_{0}^{Q}}\int_{U_{0}}\vert Df\vert^{Q}\leq\frac{C}{\varepsilon_{0}^{Q}}\int_{U_{0}}\vert D_{h}f\vert^{Q}<\infty.$$ 
This finishes the proof.
\end{proof}
\section{Sense-preserving type result for quasiregular values}\label{sec6}
In this section, we will study the sense-preserving type result for quasiregular values on Carnot group of $H$-type. The following result is analogue to \cite[Lemma 5.4]{KO221} in the setting of Carnot group. The main tools to prove Lemma 5.4  in \cite{KO221} are the truncated logarithms defined by $\vert f\vert_{\lambda}(x):=\max\{\rho\circ f( x),\lambda\}$ for every $\lambda>0,$ and the Sobolev-Poincar\'{e} inequality on Euclidean space. In the case of Carnot group of $H$-type, we use similar truncated logarithms and Lemma \ref{101209}
 to establish the following result.
\begin{lem} \label{lemma3522}
For a given $\varepsilon_{0}\in(0,1),$ we suppose that $\varepsilon_{0}/(1-\varepsilon_{0})<\varepsilon\leq\infty$ and $(1-\varepsilon_{0})(1+\varepsilon)/((1-\varepsilon_{0})\varepsilon-\varepsilon_{0})<p\leq\infty.$ Let $\mathbb{G}$ be a Carnot group of $H$-type,  $y_{0}\in\mathbb{G},$ $\Omega\subset\mathbb{G}$ be a domain and $f:\Omega\rightarrow \mathbb{G}$ be an nonconstant mapping of class $W_{loc}^{1,Q}(\Omega)$ satisfying  the following inequality
$$\vert D_{h}f(x)\vert^{Q}\leq K(x)J_{f}(x)+\Sigma(x)(\rho\circ l_{y_{0}^{-1}}\circ f(x))^{Q}$$
for almost every $x\in\Omega.$ Suppose that $K\in L_{loc}^{p}(\Omega),$  $\Sigma\in L_{loc}^{1+\varepsilon}(\Omega)$  and $K$ satisfies $(\mathcal{K},Q-1)$-spherical  condition on $\Omega.$ If $\vert D_{h}f\vert^{Q}/K(\rho\circ l_{y_{0}^{-1}}\circ f)^{Q}\in L_{loc}^{1}(\Omega),$
then $\log\rho\circ l_{y_{0}^{-1}}\circ f\in W_{loc}^{1,q_{1}}(\Omega),$ where $q_{1}=pQ/(1+p)<Q.$ In particular, $\log\rho\circ l_{y_{0}^{-1}}\circ f\in W_{loc}^{1,Q}(\Omega)$ if $K(x)\equiv K\geq1.$
\end{lem}
\begin{proof}
Without loss of generality, we here only consider the case for $y_{0}=0,$ $K\in L^{p}(\Omega),$  $\Sigma\in L^{1+\varepsilon}(\Omega)$ and $\vert D_{h}f\vert^{Q}/K(\rho\circ l_{y_{0}^{-1}}\circ f)^{Q}\in L^{1}(\Omega).$ By letting $g_{\lambda}:=\log\vert f\vert_{\lambda}:\Omega\rightarrow\mathbb{R}$ for every $\lambda>0,$ we see that $g_{\lambda}\in W^{1,Q}(\Omega).$ For every fixed $x_{0}\in\Omega,$ we chose $B=B(x_{0},r)$ which compactly contained in $\Omega.$ Since 
\begin{equation}\label{1009982887}
\int_{B}\vert\nabla_{h}g_{\lambda}\vert^{q_{1}}\leq\int_{B}\frac{\vert D_{h}f\vert^{q_{1}}}{(\rho\circ f)^{q_{1}}}=\int_{B}\left(\frac{\vert D_{h}f\vert^{Q}}{K(\rho\circ f)^{Q}}\right)^{\frac{q_{1}}{Q}}K^{\frac{q_{1}}{Q}}\leq C,
\end{equation}
where $C=C(Q,q_{1},\Omega,\mathcal{K}, \| \vert D_{h}f\vert^{Q}/K(\rho\circ l_{y_{0}^{-1}}\circ f)^{Q}\|_{L^{1}(\Omega)})$ is a constant does not depend on $\lambda,$ we see by Poincar\'{e} inequality in Lemma \ref{-11}  that 
 $$\int_{B}\vert g_{\lambda}-(g_{\lambda})_{B}\vert dx
\leq C(B)\int_{B}\vert g_{\lambda}-(g_{\lambda})_{B}\vert^{q_{1}} dx\leq C(B)\int_{B}\vert \nabla_{h}g_{\lambda}\vert^{q_{1}} dx\leq C,$$
 where $C=C(Q,q_{1},\Omega,\mathcal{K}, \| \vert D_{h}f\vert^{Q}/K(\rho\circ l_{y_{0}^{-1}}\circ f)^{Q}\|_{L^{1}(\Omega)})$ is a constant does not depend on $\lambda.$ 
 Without loss of generality, we may assume that $\rho\circ f(x)<1$ and $0<\lambda<1.$ In addition, from the proof of Lemma  \ref{101209},  we see that $\log\log\frac{1}{\rho}\circ f\in W_{loc}^{1,q_{1}}(\Omega),$ where $q_{1}=pQ/(1+p).$ Hence, $f^{-1}\{y_{0}\}$ has zero Haar measure. So we may assume that $\rho\circ f(x)$ is not identically $0$ on $B.$ Then, there exists $0<t_{0}<t_{1}<1,$ such that the set $A:=\{x\in B:t_{0}<\rho\circ f(x)<t_{1}<1\}$ has positive  measure. In addition, 
 \begin{equation}
\begin{aligned}
&\int_{B}\vert g_{\lambda}-(g_{\lambda})_{B}\vert dx\geq\int_{A}\vert g_{\lambda}-(g_{\lambda})_{B}\vert dx\\
\geq&\int_{A}(\vert (g_{\lambda})_{B}\vert-\vert g_{\lambda}\vert )dx\geq\int_{A}(\vert (g_{\lambda})_{B}\vert-\log\frac{1}{t_{0}})dx.
\end{aligned}
\end{equation}
  Hence,
  \begin{equation}\label{0009j}
  \int_{B}\vert g_{\lambda}\vert dx\leq C(m(A),B,t_{0},Q).
  \end{equation}
By monotone convergence theorem, we get that $\log\rho\circ f\in L^{1}(B).$ We next will show that $\log\rho\circ f\in  L^{q_{1}}(B).$ By Poincar\'{e} inequality   (\ref{0311888888000})  and inequalities (\ref{1009982887}) and (\ref{0009j}), we get 
\begin{equation}
\begin{aligned}
&\int_{B}\vert g_{\lambda}\vert^{q_{1}}dx\\
\leq& C(q_{1})\int_{B}\vert g_{\lambda}-(g_{\lambda})_{B}\vert^{q_{1}}dx+C(q_{1})\int_{B}\vert(g_{\lambda})_{B}\vert^{q_{1}}dx\\
\leq& C(B,q_{1})\int_{B}\vert \nabla_{h}g_{\lambda}\vert^{q_{1}}dx+C(B,q_{1})\left(\int_{B} \vert\log\rho\circ f\vert dx\right)^{q_{1}}\leq C,\\
\end{aligned}
\end{equation}
where $C$ depends only on $B,q_{1},Q $ and $\| \log\rho\circ f\|_{L^{1}(B)}.$ Hence, by  monotone convergence theorem again, we get that $\log\rho\circ f\in L^{q_{1}}(B).$ In addition,  monotone convergence Theorem  and inequality (\ref{1009982887}) make sure that 
$\nabla_{h}\log\rho\circ f\in L^{q_{1}}(B).$ Therefore, we have $\log\log\frac{1}{\rho}\circ f\in W^{1,q_{1}}(B),$ where $q_{1}=pQ/(1+p).$ This finishes the proof. 
\end{proof}

The following lemma is analogue to \cite[Lemma 4.2]{KO221} in the setting of Carnot group of $H$-type. We leave the proof for the readers.
\begin{lem} \label{98788711}
Let $\varepsilon_{0}\in(0,1)$ be a given constant. Suppose that $\mathbb{G}$ is a Carnot group of $H$-type,  $y_{0}\in\mathbb{G},$ $\Omega\subset\mathbb{G}$ is a domain and $f:\Omega\rightarrow \mathbb{G}$ is an nonconstant mapping of class $W_{loc}^{1,Q}(\Omega)$ satisfying  the following inequality
$$\vert D_{h}f(x)\vert^{Q}\leq K(x)J_{f}(x)+\Sigma(x)(\rho\circ l_{y_{0}^{-1}}\circ f(x))^{Q}$$
for almost every $x\in\Omega.$ Suppose that $K\in L_{loc}^{p}(\Omega),$  $\Sigma\in L_{loc}^{1+\varepsilon}(\Omega),$ where $(1-\varepsilon_{0})(1+\varepsilon)/((1-\varepsilon_{0})\varepsilon-\varepsilon_{0})<p\leq\infty$ and $\varepsilon_{0}/(1-\varepsilon_{0})<\varepsilon\leq\infty,$  and $K$ satisfies $(\mathcal{K},Q-1)$-spherical condition on $\Omega.$ Suppose that $U$ is a non-empty component of $f^{-1}B(y_{0},\varepsilon)$ for $\varepsilon>0$ is small enough such that $\overline{U}\subset\Omega$ with $m(U)<C,$ where $C=C(Q,K,\Sigma,\Omega)>0$ is a constant given by Lemma \ref{987887}. If  ${\rm deg}(f,U)=0,$ then we have
\begin{equation}
\int_{U\cap f^{-1}B(y_{0},r)}J_{f}=0
\end{equation}
for every $r\in(0,\varepsilon).$ 
\end{lem}

By Lemmas \ref{987887} and \ref{98788711}, we have the following Lemma \ref{lemma000829992}, which is a Carnot group of $H$-type version of \cite[Lemma 4.4]{KO222}. The main tool to obtain Lemma 4.4 in \cite{KO222} is a beautifully Logarithmic isoperimetric inequality, which established by  Kangasniemi and Onninen in \cite[Lemma 6.1]{KO221}.  The proof of this Logarithmic isoperimetric inequality
in \cite[Lemma 6.1]{KO221} used  the following Sobolev embedding theorem on spheres 
\begin{equation}\label{shpere1}
{\rm osc}(f,\partial B(x,r))\leq C_{n}r^{\frac{1}{n}}\left(\int_{\partial B(x,r)}\vert Df\vert^{n}\right)^{\frac{1}{n}},
\end{equation}
  on the setting of Euclidean space. It seems that 
 it is interesting to obtain the analogous Sobolev embedding theorem on spheres (\ref{shpere1}) on the setting of Carnot group. In \cite[Corollary  1]{Vod96},  Vodoptyanov showed that if  
 $f\in W_{loc}^{1,Q}(\Omega),$ then the following inequality
 \begin{equation}\label{shpere2}
{\rm osc}(f,\partial B(x,r))\leq C_{Q}r^{\frac{1}{Q}}\left(\int_{\partial B(x,r)}M_{c_{1} r}^{Q}(\vert D_{h}f\vert)d\sigma\right)^{\frac{1}{Q}}
\end{equation}
  holds for almost every $r\in(0,{\rm dist}(x,\partial\Omega)).$ Here, $c_{1}$ is the constant in the triangle inequality in (\ref{00oo}) and $M_{\delta}(g)$ is the maximal function of a locally summable function $g$ which  defined by
 $$M_{\delta}g(x)=\sup\Bigg\{\frac{1}{m(B(r))}\int_{B(r)}\vert g(x)\vert dx: r\leq\delta\Bigg\}.$$
 Yet, there are some differences between the integrals in right side of inequality  (\ref{shpere2}) and inequality  (\ref{shpere1}). We here use another  method  to prove the following Lemma \ref{lemma000829992}. Our method bypasses the use of  Sobolev embedding theorem on spheres on the setting of Carnot group of $H$-type.
\begin{lem} \label{lemma000829992}
 Let $K\geq1$ be a given constant. Suppose that $\mathbb{G}$ is a Carnot group of $H$-type,  $y_{0}\in\mathbb{G},$ $\Omega\subset\mathbb{G}$ is a domain and $f:\Omega\rightarrow \mathbb{G}$ is an nonconstant mapping of class $W_{loc}^{1,Q}(\Omega)$ satisfying  the following inequality
$$\vert D_{h}f(x)\vert^{Q}\leq KJ_{f}(x)+\Sigma(x)(\rho\circ l_{y_{0}^{-1}}\circ f(x))^{Q}$$
for almost every $x\in\Omega.$ Suppose that $\Sigma\in L_{loc}^{p}(\Omega)$ for some $p>1$  and $U$ is a non-empty component of $f^{-1}B(y_{0},\varepsilon)$ for $\varepsilon>0$ is small enough such that $\overline{U}\subset\Omega$ with $m(U)<C,$ where $C=C(Q,K,\Sigma,\Omega)>0$ is a constant given by Lemma \ref{987887}. Then, ${\rm deg}\;(f,U)>0.$
\end{lem}
\begin{proof}
According to Lemma \ref{987887}, we see that ${\rm deg}\;(f,U)\geq0.$ We further prove that ${\rm deg}(f,U)\neq0.$ We prove this by contradiction if ${\rm deg}(f,U)=0$ under the conditions of Lemma \ref{lemma000829992}. We split the details of proof into following three steps: 

${\rm \mathbf{Step}\;\mathbf{1}}$:  We conclude that
\begin{equation}\label{Lemmamma}
\int_{U}\frac{\vert D_{h}f\vert^{Q}}{(\rho\circ l_{y_{0}^{-1}}\circ f)^{Q}}<\infty\;\;\;\;\;\;\;\;{\rm and }\;\;\;\;\;\;\;\int_{U}\frac{J_{f}}{(\rho\circ l_{y_{0}^{-1}}\circ f)^{Q}}=0.
\end{equation}
The proof of those results are similar to the one given by \cite[Lemma 4.3]{KO222}. For the convenience of readers, we here provide the detailed proof of inequalities (\ref{Lemmamma}). By splitting $J_{f}$ into its positive and negative parts $J_{f}=J_{f}^{+}-J_{f}^{-},$ then we have
$$\frac{1}{K}\frac{\vert D_{h}f(x)\vert^{Q}}{(\rho\circ l_{y_{0}^{-1}}\circ f)^{Q}}+\frac{J_{f}^{-}}{(\rho\circ l_{y_{0}^{-1}}\circ f)^{Q}}\leq \frac{J_{f}^{+}}{(\rho\circ l_{y_{0}^{-1}}\circ f)^{Q}}+\frac{\Sigma}{K},$$
which implies that 
$$\frac{J_{f}^{-}}{(\rho\circ l_{y_{0}^{-1}}\circ f)^{Q}}\leq \frac{\Sigma}{K}.$$
Hence, we get
$$\int_{U}\frac{J_{f}^{-}}{(\rho\circ l_{y_{0}^{-1}}\circ f)^{Q}}\leq \int_{U}\frac{\Sigma}{K}<\infty.$$
Now, let $U_{r}=U\cap f^{-1}B(y_{0},r)$ for every $r\in(0,\varepsilon).$ Then, according to Lemma \ref{98788711}, we have
$$\int_{U_{r}}J_{f}^{+}=\int_{U_{r}}J_{f}^{-},$$
which gives that 
$$\int_{0}^{\varepsilon}Qr^{-Q-1}\int_{U_{r}}J_{f}^{+}(x)dxdr=\int_{0}^{\varepsilon}Qr^{-Q-1}\int_{U_{r}}J_{f}^{-}(x)dxdr.$$
By Fubini-Tonelli theorem, we get
$$\int_{U}J_{f}^{+}(x)\int_{\rho\circ l_{y_{0}^{-1}}\circ f(x)}^{\varepsilon}Qr^{-Q-1}drdx=\int_{U}J_{f}^{-}(x)\int_{\rho\circ l_{y_{0}^{-1}}\circ f(x)}^{\varepsilon}Qr^{-Q-1}drdx.$$
Namely,
$$\int_{U}\frac{J_{f}^{+}}{(\rho\circ l_{y_{0}^{-1}}\circ f)^{Q}}-\int_{U}\frac{J_{f}^{+}}{\varepsilon^{Q}}=\int_{U}\frac{J_{f}^{-}}{(\rho\circ l_{y_{0}^{-1}}\circ f)^{Q}}-\int_{U}\frac{J_{f}^{-}}{\varepsilon^{Q}}.$$
Combing this with Lemma \ref{lomlm} and ${\rm deg}\;(f,U)=0,$ we have
$$\int_{U}\frac{J_{f}^{+}}{(\rho\circ l_{y_{0}^{-1}}\circ f)^{Q}}=\int_{U}\frac{J_{f}^{-}}{(\rho\circ l_{y_{0}^{-1}}\circ f)^{Q}}.$$
So we have
$$\int_{U}\frac{J_{f}}{(\rho\circ l_{y_{0}^{-1}}\circ f)^{Q}}=0$$
and hence 
\begin{equation}\label{00001010101010190091}
\int_{U}\frac{1}{K}\frac{\vert D_{h}f\vert^{Q}}{(\rho\circ l_{y_{0}^{-1}}\circ f)^{Q}}\leq \int_{U}\frac{J_{f}}{(\rho\circ l_{y_{0}^{-1}}\circ f)^{Q}}+\int_{U}\frac{\Sigma}{K}=\int_{U}\frac{\Sigma}{K}<\infty.
\end{equation}
Hence, the proof of inequalities (\ref{Lemmamma}) is finished.

${\rm \mathbf{Step}\;\mathbf{2}}$: Showing that there exists $\beta_{1}>1,$ such that $$\frac{\vert D_{h}\rho\circ f\vert\vert D_{h}f\vert}{\rho\circ f}\in L_{loc}^{\beta_{1} Q}(U).$$

By donating $g=\log\rho\circ l_{y_{0}^{-1}}\circ f:U\rightarrow\mathbb{R}.$ Then, the inequality (\ref{00001010101010190091}) and Lemma \ref{lemma3522} provides us that $g\in W_{loc}^{1,Q}(U).$ Without loss of generality, we may assume that $g\in W^{1,Q}(U).$ For every fixed $x\in U,$ we let $0<r<2r<{\rm dist}(x,\partial U).$ Let $\varphi\in \mathcal{C}_{0}^{\infty}(B(x,2r))$ such that $\vert D_{h}\varphi \vert\leq1/r$ on $B(x,2r)$ and $\varphi=1$ on $B(x,r).$
By considering the following function
\begin{equation}
  \psi_{\varepsilon}(t)=\left\{
\begin{aligned}
&\vert\log(\varepsilon+t)-c\vert,      &   &    {\rm if}\;{t\geq a},\\
&\vert\log(\varepsilon+a)-c\vert,       &      &{\rm if}\; {t\leq a.}
\end{aligned} \right.
\end{equation}
Then, according to Lemma \ref{l00}, we get
\begin{equation}
\begin{aligned}
&\abs{\int_{B(x,2r)}\varphi\frac{\vert\nabla_{h}\rho\circ f\vert^{Q}}{(\rho\circ f)^{Q-1}(\varepsilon+\rho\circ f)}J_{f}}\\
\leq& 
\frac{C}{r}\abs{\int_{B(x,2r)}\frac{\vert\nabla_{h}\rho\circ f\vert^{Q-1}\vert D_{h}f\vert^{Q-1}}{(\rho\circ f)^{Q-1}} \vert\log(\varepsilon+\rho\circ f)-c\vert}.\\
\end{aligned}
\end{equation}
Then, by dominated convergence theorem, we have
\begin{equation}
\begin{aligned}
&\abs{\int_{B(x,2r)}\varphi\frac{\vert\nabla_{h}\rho\circ f\vert^{Q}}{(\rho\circ f)^{Q}}J_{f}}\leq\frac{C}{r}\abs{\int_{B(x,2r)}\frac{\vert\nabla_{h}\rho\circ f\vert^{Q-1}\vert D_{h}f\vert^{Q-1}}{(\rho\circ f)^{Q-1}} \vert\log(\rho\circ f)-c\vert}.\\
\end{aligned}
\end{equation}
By H\"{o}lder inequality and an embedding inequality \cite[Theorem 3.2]{VI07}, we get 
\begin{equation}
\begin{aligned}
&\abs{\int_{B(x,2r)}\varphi\frac{\vert\nabla_{h}\rho\circ f\vert^{Q}}{(\rho\circ f)^{Q}}J_{f}} \leq\frac{C}{r}\abs{\int_{B(x,2r)}\frac{\vert\nabla_{h}\rho\circ f\vert^{Q-1}\vert D_{h}f\vert^{Q-1}}{(\rho\circ f)^{Q-1}} \vert\log(\rho\circ f)-c\vert}\\
\leq& 
\frac{C}{r}\left(\int_{B(x,2r)}\left(\frac{\vert\nabla_{h}\rho\circ f\vert\vert D_{h}f\vert}{(\rho\circ f)}\right)^{\frac{Q^{2}}{Q+1}}\right)^{\frac{Q^{2}-1}{Q^{2}}} 
\cdot\left(\int_{B(x,2r)}\vert\log(\rho\circ f)-c\vert^{Q^{2}}\right)^{\frac{1}{Q^{2}}}\\
\leq& 
\frac{C}{r}\left(\int_{B(x,2r)}\left(\frac{\vert\nabla_{h}\rho\circ f\vert\vert D_{h}f\vert}{(\rho\circ f)}\right)^{\frac{Q^{2}}{Q+1}}\right)^{\frac{Q^{2}-1}{Q^{2}}} 
\cdot\left(\int_{B(x,2r)}\left(\frac{\vert\nabla_{h}\rho\circ f\vert\vert D_{h}f\vert}{(\rho\circ f)}\right)^{\frac{Q^{2}}{Q+1}}\right)^{\frac{1+Q}{Q^{2}}}\\
 =& 
\frac{C}{r}\left(\int_{B(x,2r)}\left(\frac{\vert\nabla_{h}\rho\circ f\vert\vert D_{h}f\vert}{(\rho\circ f)}\right)^{\frac{Q^{2}}{Q+1}}\right)^{\frac{Q+1}{Q}}.
\end{aligned}
\end{equation}
Therefore, we get 
\begin{equation}
\begin{aligned}
&\int_{B(x,r)}\left(\frac{\vert\nabla_{h}\rho\circ f\vert\vert D_{h}f\vert}{(\rho\circ f)}\right)^{Q}\\
\leq& \int_{B(x,2r)}\varphi\left(\frac{\vert\nabla_{h}\rho\circ f\vert\vert D_{h}f\vert}{(\rho\circ f)}\right)^{Q}\\
\leq& K\int_{B(x,2r)}\varphi\frac{\vert\nabla_{h}\rho\circ f\vert^{Q}}{(\rho\circ f)^{Q}}J_{f}+\int_{B(x,2r)}\Sigma\\
\leq& 
\frac{C(K)}{r}\left(\int_{B(x,2r)}\left(\frac{\vert\nabla_{h}\rho\circ f\vert\vert D_{h}f\vert}{(\rho\circ f)}\right)^{\frac{Q^{2}}{Q+1}}\right)^{\frac{Q+1}{Q}}+\int_{B(x,2r)}\Sigma.\\
\end{aligned}
\end{equation}
Namely,
\begin{equation}\label{e14}
\begin{aligned}
&\int_{B(x,r)}\!\!\!\!\!\!\!\!\!\!\!\!\!\!\!\!\!\!\!\!\; {}-{} \,\,\,\,\,\,\,\,\left(\frac{\vert\nabla_{h}\rho\circ f\vert\vert D_{h}f\vert}{(\rho\circ f)}\right)^{Q}\leq C(K)\left(\int_{B(x,2r)}\!\!\!\!\!\!\!\!\!\!\!\!\!\!\!\!\!\!\!\!\!\!\!\; {}-{} \,\,\,\,\,\,\,\,\,\,\,\left(\frac{\vert\nabla_{h}\rho\circ f\vert\vert D_{h}f\vert}{(\rho\circ f)}\right)^{\frac{Q^{2}}{Q+1}}\right)^{\frac{Q+1}{Q}}\\
&\,\,\,\,\,\,\,\,\,\,\,\,\,\,\,\,\,\,\,\,\,\,\,\,\,\,\,\,\,\,\,\,\,\,\,\,\,\,\,\,\,\,\,\,\,\,\,\,\,\,\,\,\,\,\,\,\,\,\,\,\,\,\,\,\,\,\,\,\,\,\,\,\,\,\,\,\,\,\,\,\,\,\,\,\,\,\,\,\,\,\,\,\,\,\,\,\,\,\,
\,\,\,\,\,\,\,\,\,\,\,\,+
\int_{B(x,2r)}\!\!\!\!\!\!\!\!\!\!\!\!\!\!\!\!\!\!\!\!\!\!\!\; {}-{} \,\,\,\,\,\,\,\,\,\,\,\Sigma.\\
\end{aligned}
\end{equation}
Let $$q=\frac{Q+1}{Q}, r_{0}=(1+\varepsilon)q>q, g=\left(\frac{\vert\nabla_{h}\rho\circ f\vert\vert D_{h}f\vert}{(\rho\circ f)}\right)^{\frac{Q^{2}}{Q+1}}\;\;\;{\rm and}\;\;\;f=\Sigma^{\frac{Q}{Q+1}}.$$
Then, $f\in L_{loc}^{r_{0}}(\Omega)$ and 
\begin{equation}
\begin{aligned}
&\int_{B(x,r)}\!\!\!\!\!\!\!\!\!\!\!\!\!\!\!\!\!\!\!\!\; {}-{} \,\,\,\,\,\,\,\,\,\,\,g^{q}\leq C(K)\left[\left(\int_{B(x,2r)}\!\!\!\!\!\!\!\!\!\!\!\!\!\!\!\!\!\!\!\!\!\!\!\; {}-{} \,\,\,\,\,\,\,\,\,\,\,\,\,\,g\right)^{q}+
\int_{B(x,2r)}\!\!\!\!\!\!\!\!\!\!\!\!\!\!\!\!\!\!\!\!\!\!\!\; {}-{} \,\,\,\,\,\,\,\,\,\,\,\,\,f^{q}\right].\\
\end{aligned}
\end{equation}
Hence, by a Gehring's lemma on metric measure space \cite[Theorem 3.3]{Gol05}, we see that   there exists $\beta_{1}>1,$ such that $\vert D_{h}\rho\circ f\vert\vert D_{h}f\vert/\rho\circ f\in L_{loc}^{\beta_{1} Q}(U).$ This finishes the proof.

${\rm \mathbf{Step}\;\mathbf{3}}$: Showing that $$g=\log\rho\circ l_{y_{0}^{-1}}\circ f\in L_{loc}^{\infty}(U).$$

By $\mathbf{Step\,2}$, we see that $\vert D_{h}\rho\circ f\vert\vert D_{h}f\vert/\rho\circ f\in L_{loc}^{\beta_{1} Q}(U).$ Without loss of generality, we assume that $\vert D_{h}\rho\circ f\vert\vert D_{h}f\vert/\rho\circ f\in L^{\beta_{1} Q}(U).$ Then, by virtue of  H\"{o}lder inequality, we get
\begin{equation}
\begin{aligned}
&\int_{B(x,r)}\vert \nabla_{h}g\vert^{Q}\\
\leq&\int_{B(x,r)}\frac{\vert D_{h}\rho\circ f\vert^{Q}\vert \nabla_{h}f\vert^{Q}}{(\rho\circ f)^{Q}}\\
\leq& 
\left(\int_{B(x,r)}\left(\frac{\vert D_{h}\rho\circ f\vert\vert \nabla_{h}f\vert}{(\rho\circ f)}\right)^{\beta_{1} Q}\right)^{\frac{1}{\beta_{1}}} \cdot r^{\frac{\beta_{1} Q}{\beta_{1}-1}}\\
\leq& C \cdot r^{\frac{\beta_{1} Q}{\beta_{1}-1}},\\
\end{aligned}
\end{equation}
where $C=\vert\vert\vert D_{h}\rho\circ f\vert\vert D_{h}f\vert/\rho\circ f\vert\vert_{L^{\beta_{1} Q}(U)}.$
Since $g\in W^{1,Q}(U),$ by using  Proposition \ref{09888888111} to $g:\Omega:\rightarrow \mathbb{R}$ for the case that $\alpha=\beta_{1}Q/(\beta_{1}-1)>0$ and $\beta=0,$ we see that $g$ is local H\"{o}lder continuous on $U$. Hence, we completes the proof that $g=\log\rho\circ l_{y_{0}^{-1}}\circ f\in L_{loc}^{\infty}(U).$ 

With above three steps, we see that the conclusion of Lemma \ref{lemma000829992} is correct. This is because  $g=\log\rho\circ l_{y_{0}^{-1}}\circ f\in L_{loc}^{\infty}(U)$ is impossible as $x_{0}\in U$ and hence $\lim_{x\rightarrow x_{0}}\log\rho\circ l_{y_{0}^{-1}}\circ f(x)=-\infty.$ This completely finished the proof of Lemma \ref{lemma000829992}.
\end{proof}

With Lemma \ref{lemma000829992} and the definition of local index. We can proceed to prove the following result.
\begin{thm}\label{900177}
 Let  $K\geq1$ be a given constant. Suppose that $\mathbb{G}$ is a Carnot group of $H$-type,  $y_{0}\in\mathbb{G},$ $\Omega\subset\mathbb{G}$ is a domain and $f:\Omega\rightarrow \mathbb{G}$ is an nonconstant mapping of class $W_{loc}^{1,Q}(\Omega)$ satisfying  the following inequality
$$\vert D_{h}f(x)\vert^{Q}\leq KJ_{f}(x)+\Sigma(x)(\rho\circ l_{y_{0}^{-1}}\circ f(x))^{Q}$$
for almost every $x\in\Omega.$ Suppose that $\Sigma\in L_{loc}^{p}(\Omega)$ for some $p>1,$ then for every $x\in f^{-1}\{y_{0}\},$ the local index $i(x,f)$ is well defined and is always positive.
\end{thm}
\begin{proof}
According to Lemma \ref{lemma000829992}, we see that there exists a $\varepsilon>0,$ such that there is an non-empty $x$-component $U_{\varepsilon}$ of every $x\in f^{-1}\{y_{0}\},$  with which $U_{\varepsilon}$ is compactly contained in $\Omega$ and satisfying ${\rm deg}(f,U_{\varepsilon})>0.$ In addition, according to Theorem \ref{theorem9009},  we see that $f^{-1}\{y_{0}\}$ is discrete.
 Hence, we can take $\varepsilon$ to be small enough, such that $\overline{U_{\varepsilon}}\cap f^{-1}\{y_{0}\}=\{x\}.$ Hence, the local index $i(x,f)$ is well defined and $i(x,f)={\rm deg}(f,U_{\varepsilon})>0.$ This finishes the proof.
\end{proof}

\section{Openness type result for quasiregular values}\label{sec7}
In this section, we will show the  openness type result for quasiregular values on Carnot group of $H$-type. We have the following result.
\begin{thm}
Let $K\geq1$ be a given constant. Suppose that $\mathbb{G}$ is a Carnot group of $H$-type,  $y_{0}\in\mathbb{G},$ $\Omega\subset\mathbb{G}$ is a domain and $f:\Omega\rightarrow \mathbb{G}$ is an nonconstant mapping of class $W_{loc}^{1,Q}(\Omega)$ satisfying  the following inequality
$$\vert D_{h}f(x)\vert^{Q}\leq KJ_{f}(x)+\Sigma(x)(\rho\circ l_{y_{0}^{-1}}\circ f(x))^{Q}$$
for almost every $x\in\Omega.$ Suppose that $\Sigma\in L_{loc}^{p}(\Omega)$ for some $p>1,$ then every neighborhood $V$ of every  $x\in f^{-1}\{y_{0}\},$ we have  $y_{0}\in{\rm int}f(V).$
\end{thm}
\begin{proof}
We must prove that for every neighborhood $V$ of every  $x\in f^{-1}\{y_{0}\},$ there exists a $\varepsilon>0,$ such that $B(y_{0},\varepsilon)\subset f(V).$ If there is no such $\varepsilon,$ then there must exist 
$x_{0}\in f^{-1}\{y_{0}\},$ such that for an neighborhood $V_{0}$ of $x_{0},$ we have $B(y_{0},\varepsilon)\setminus f(V_{0})\neq \emptyset$ for all $\varepsilon>0.$ 

On the other hand, according to Lemma \ref{lemma000829992}, we see that there exists a $\varepsilon_{1}>0,$ such that an non-empty $x_{0}$-component $U_{\varepsilon_{1}}$ of $f^{-1}B(y_{0},\varepsilon_{1})$ satisfying that $\overline{U_{\varepsilon_{1}}}\subset\Omega,$ $\overline{U_{\varepsilon_{1}}}\cap f^{-1}\{y_{0}\}=\{x_{0}\}$ and ${\rm deg}(f,U_{\varepsilon_{1}})>0.$ 

Now, for  $0<\varepsilon<\varepsilon_{1},$ suppose that $U\varepsilon$ is the $x_{0}$-componentvof $f^{-1}B(y_{0},\varepsilon)$ such that $U_{\varepsilon}$ is compactly contained in $\Omega.$ Then, we must have the equation
$f(U_{\varepsilon})=B(y_{0},\varepsilon).$ This is because for every $y_{1}\in B(y_{0},\varepsilon),$ we must have ${\rm deg}(f,U_{\varepsilon})={\rm deg}(f,U_{\varepsilon},y_{1})>0$ according to Theorem \ref{900177}.
Now, if $y_{1}\notin f(U_{\varepsilon}),$ then by \cite[Theorem 2.1]{FG95}, we must ${\rm deg}(f,U_{\varepsilon})\leq0.$ Hence, we must have $y_{1}\in f(U_{\varepsilon}).$ Namely, $f(U_{\varepsilon})=B(y_{0},\varepsilon).$

With above conclusions,  we have $B(y_{0},\varepsilon)\setminus f(V_{0})=f(U_{\varepsilon})\setminus f(V_{0})\neq \emptyset.$ So we have
$U_{\varepsilon}\setminus V_{0}\neq \emptyset.$ Hence, $\overline{U_{\varepsilon}}\setminus V_{0}$ is non-empty compact subset of $\Omega,$ where $0<\varepsilon<\varepsilon_{1},$ which implies that their
intersection is non-empty. This is contradictory since  $$\cap_{0<\varepsilon<\varepsilon_{1}}\overline{U_{\varepsilon}}\setminus V_{0}=\{x_{0}\}\setminus V_{0}= \emptyset.$$ This finishes the proof.
\end{proof}

%%%%%%%%%%%%%%%%%%%%%%%%%%%%%%%%%%%%%%%%%%%%%%%%%%%%%%%%%%%%%%%%%%%%%%%%

%%%%%%%%%%%%%%%%%%%%%%%%%%%%%%%%%%%%%%%%%%%%%%%%%%%%%%%%%%%%%%%%%%%%%%%%

\section{Acknowledge}
This work was supported by Guangdong Basic and Applied Basic Research Foundation (No. 2022A1515110967 and No. 2023A1515011809) and Supporting Foundation of Shenzhen Polytechnic University (No. 6024310047K).

%%%%%%%%%%%%%%%%%%%%%%%%%%%%%%%%%%%%%%%%%%%%%%%%%%%%%%%%%%%%%%%%%%%%%%%%

\begin{flushleft}
 $\mathbf{Data\;availability}.$  This paper has no associated data.
\end{flushleft}

\end{document}